\documentclass[12pt, reqno]{amsproc}
\usepackage[utf8]{inputenc}
\usepackage[T2A]{fontenc}
\usepackage[english]{babel}

\usepackage[margin = 2.45cm]{geometry}
\usepackage{enumitem}
\usepackage[svgcolors, dvipsnames]{xcolor}
\usepackage{graphicx}
\usepackage{mathtools}
\usepackage[all]{xy}
\usepackage{amssymb, amsmath, amscd, amsthm}
\usepackage[colorlinks=true]{hyperref}
\hypersetup{urlcolor=blue, citecolor=red}

\thanks{This work was supported by a grant from the Simons Foundation (1030291, 1290607, S.I.M.).}

\title[Vector bundle automorphisms preserving Morse-Bott foliations]{Vector bundle automorphisms \\ preserving Morse-Bott foliations}

\author{Sergiy Maksymenko}
\address{Algebra and Topology Department, Institute of Mathematics NAS of Ukraine \\
Teresh\-chen\-kivska str., 3, Kyiv, 01024, Ukraine}
\email{maks@imath.kiev.ua}


\newcommand\mycolor[1]{}

\setlist[enumerate]{itemsep=0.3ex, topsep=0.3ex, label={\rm(\arabic*)}}
\setlist[itemize]{itemsep=0.3ex, topsep=0.3ex, leftmargin=4ex}

\newtheorem{lemma}[subsection]{Lemma}

\newtheorem{subtheorem}[subsubsection]{Theorem}
\newtheorem{sublemma}[subsubsection]{Lemma}

\newtheorem{subcorollary}[subsubsection]{Corollary}

\newtheorem{subremark}[subsubsection]{Remark}
\newtheorem{subexample}[subsubsection]{Example}
\newtheorem{subdefinition}[subsubsection]{Definition}
\newtheorem{subaddendum}[subsubsection]{Addendum}

\makeatletter
\@addtoreset{subsection}{section}
\@addtoreset{equation}{section}
\@addtoreset{figure}{section}
\@addtoreset{table}{section}
\makeatother

\makeatletter
\newcommand\testshape{family=\f@family; series=\f@series; shape=\f@shape.}
\def\myemphInternal#1{\if n\f@shape%
\begingroup\itshape #1\endgroup\/%
\else\begingroup\sf\itshape\small #1\endgroup%
\fi}
\def\myemph{\futurelet\testchar\MaybeOptArgmyemph}
\def\MaybeOptArgmyemph{\ifx[\testchar \let\next\OptArgmyemph
                 \else \let\next\NoOptArgmyemph \fi \next}
\def\OptArgmyemph[#1]#2{\index{#1}\myemphInternal{#2}}
\def\NoOptArgmyemph#1{\myemphInternal{#1}}
\makeatother

\newcommand\monoArrow{\lhook\joinrel\rightarrow}

\newcommand\xmonoArrow[1]{\lhook\joinrel\xrightarrow{~#1~}}

\newcommand\epiArrow{\rightarrow\!\!\!\!\!\to}

\newcommand\xepiArrow[1]{\xrightarrow{#1}\!\!\!\!\!\to}

\newcommand\amatr[4]{\left(\!\begin{smallmatrix}#1\ &#2\\[0.5mm] #3\ &#4 \end{smallmatrix}\!\right)}

\newcommand\Aman{A}
\newcommand\Bman{B}
\newcommand\Cman{C}
\newcommand\Dman{D}
\newcommand\Eman{E}

\newcommand\Kman{K}

\newcommand\Mman{M}

\newcommand\Uman{U}
\newcommand\Vman{V}
\newcommand\Wman{W}
\newcommand\Xman{X}
\newcommand\Yman{Y}

\newcommand\calF{\mathcal{F}}

\newcommand\calP{\mathcal{P}}
\newcommand\calQ{\mathcal{Q}}
\newcommand\calR{\mathcal{R}}

\newcommand\bC{\mathbb{C}}

\newcommand\bN{\mathbb{N}}

\newcommand\bR{\mathbb{R}}
\newcommand\bZ{\mathbb{Z}}

\newcommand\UU{\mathcal{U}}
\newcommand\VV{\mathcal{V}}
\newcommand\WW{\mathcal{W}}

\newcommand\id{\mathrm{id}}          
\newcommand\supp{\mathrm{supp\,}}    
\newcommand\restr[2]{#1\vert_{#2}}   
\newcommand\nrm[1]{\|#1\|}

\newcommand\eps{\varepsilon}                   
\newcommand\GL{\mathrm{GL}}

\newcommand\SO{\mathrm{SO}}
\newcommand\Ort{\mathrm{O}}

\newcommand\unit[1]{e_{#1}}
\newcommand\GLR[1]{\mathrm{GL}_{#1}(\bR)}
\newcommand\MatR[1]{\mathrm{M}_{#1}(\bR)}

\newcommand\gen[1]{\langle#1\rangle}
\newcommand\Aut{\mathrm{Aut}}       
\newcommand\Diff{\mathcal{D}}       
\newcommand\End{\mathrm{End}}       
\newcommand\Emb{\mathrm{Emb}}       
\newcommand\LDiff{\mathcal{L}} 
\newcommand\DiffId{\Diff_{\id}}     

\newcommand\Cr[1]{\mathcal{C}^{#1}}
\newcommand\Cinfty{\mathcal{C}^{\infty}}
\newcommand\Crm[3]{\Cr{#1}\!\left(#2,#3\right)}
\newcommand\Cont[2]{\Crm{0}{#1}{#2}}                         
\newcommand\Ci[2]{\Crm{\infty}{#1}{#2}}               
\newcommand\Chid[2]{\mathcal{C}_{0}^{\infty}\left(#1,#2\right)}   

\newcommand\fixsymbol{\mathrm{fix}}
\newcommand\invsymbol{\mathrm{inv}}
\newcommand\nbsymbol{\mathrm{nb}}
\newcommand\folsymbol{*}

\newcommand\DiffInv[3][\empty]{\Diff_{\invsymbol}(#2,#3\ifx\empty #1\relax\else,#1\fi)}

\newcommand\DiffFix[3][\empty]{\Diff_{\fixsymbol}(#2,#3\ifx\empty #1\relax\else,#1\fi)}

\newcommand\DiffNb[3][\empty]{\Diff_{\nbsymbol}(#2,#3\ifx\empty #1\relax\else,#1\fi)}

\newcommand\DiffHFix[3][\empty]{\Diff^{0}_{\fixsymbol}(#2,#3\ifx\empty #1\relax\else,#1\fi)}

\newcommand\DiffHNb[3][\empty]{\Diff^{0}_{\nbsymbol}(#2,#3\ifx\empty #1\relax\else,#1\fi)}

\newcommand\DiffPlusFix[3][\empty]{\Diff^{+}_{\fixsymbol}(#2,#3\ifx\empty #1\relax\else,#1\fi)}

\newcommand\FDiff[2][\empty]{\Diff^{\folsymbol}(#2\ifx\empty #1\relax\else,#1\fi)}
\newcommand\FDiffFix[2][\empty]{\Diff^{\folsymbol}_{\fixsymbol}(#2\ifx\empty #1\relax\else,#1\fi)}
\newcommand\FDiffA[2][\empty]{\Diff^{=}(#2\ifx\empty #1\relax\else,#1\fi)}

\newcommand\VBAut[2][\empty]{\GL(#2\ifx\empty #1\relax\else,#1\fi)}

\newcommand\DiffLP{\Diff}  
\newcommand\DiffLPInv[3][\empty]{\DiffLP_{inv}(#2,#3\ifx\empty#1\relax\else,#1\fi)}

\newcommand\DiffLPFix[3][\empty]{\DiffLP_{fix}(#2,#3\ifx\empty#1\relax\else,#1\fi)}

\newcommand\DiffLPNb[3][\empty]{\DiffLP_{nb}(#2,#3\ifx\empty#1\relax\else,#1\fi)}

\newcommand\GLV[2][\empty]{\mathrm{GL}(#2\ifx\empty #1\relax\else;#1\fi)}
\newcommand\GLVFix[3][\empty]{\mathrm{GL}(#2,#3\ifx\empty #1\relax\else;#1\fi)}

\newcommand\func{f}
\newcommand\gfunc{g}
\newcommand\dif{h}
\newcommand\gdif{g}

\newcommand\px{x}
\newcommand\py{y}
\newcommand\pz{z}
\newcommand\pu{u}
\newcommand\pv{v}

\newcommand\pui[1]{\pu_{#1}}
\newcommand\pvi[1]{\pv_{#1}}

\newcommand\Circle{S^1}

\newcommand\dker{$0$-kernel}
\newcommand\zeroker[1]{\ker_{0}(#1)}

\newcommand\NGOOD[1]{{\mycolor{NavyBlue}#1-isolating and #1-scalable}}

\newcommand\divhom{{\mycolor{Green}\alpha}}

\newcommand\Dr[1]{\Dman^{n}_{#1}}
\newcommand\rr{{\mycolor{red}r}}

\newcommand\hFoliation{\hat{\Foliation}}

\newcommand\Hhmt{{\mycolor{Green}\mathbf{H}}}

\newcommand\Ghmt{{\mycolor{Green}\mathbf{G}}}

\newcommand\Alin{{\mycolor{orange}A}}
\newcommand\Blin{{\mycolor{orange}B}}
\newcommand\Clin{{\mycolor{orange}C}}

\newcommand\homdata[1]{\chi(\func)}

\newcommand\term[2][\empty]{\myemph[#1]{#2}}

\newcommand\fSing{\Sigma}

\newcommand\vbp{{\mycolor{blue}p}}

\newcommand\Foliation{\mathcal{F}}

\newcommand\aConst{0.2}
\newcommand\bConst{0.8}

\newcommand\vvrt[2][\empty]{\mathsf{Vert}\ifx\empty\relax\else_{#1}\fi#2}  
\newcommand\tang[2][\empty]{\mathsf{T}\ifx\empty\relax\else_{#1}\fi#2}  
\newcommand\tfib[1]{\tang[\!\mathsf{fib}]{#1}}                          

\newcommand\ATor{\mathbf{T}}

\newcommand\dATor{\partial\ATor}

\newcommand\tRestr{{\mycolor{red}\mathsf{r}}}   
\newcommand\tFibMap{{\mycolor{Green}\mathsf{t}}}

\newcommand\arotat[1]{R_{#1}}
\newcommand\rotat[2]{\arotat{(#1,#2)}}
\newcommand\hDtwist{d}
\newcommand\hLambda{\lambda}
\newcommand\hMu{\mu}
\newcommand\hTau{\tau}

\newcommand\pbs{{\mycolor{Green}\py}}

\newcommand\leaf{\omega}

\newcommand\EndES{\End(\Eman,\fSing)}
\newcommand\GLE{\GL(\Eman)}
\newcommand\GLES{\GL(\Eman,\fSing)}
\newcommand\GLESGG[1]{\GL(\Eman,\fSing;#1)}

\newcommand\OBall[1]{\Bman_{#1}}

\newcommand\sectrVB{\sigma}

\newcommand\NbhEnd{{\mycolor{blue}\mathcal{N}}}
\newcommand\NbhIdE{{\mycolor{blue}\mathcal{M}}}
\newcommand\NbhEndp{{\mycolor{blue}\mathcal{N}'}}

\newcommand\NbhGIdE{{\mycolor{blue}\mathcal{L}}}
\newcommand\NbhGIdEp{{\mycolor{blue}\mathcal{L}'}}

\newcommand\GDistr{\mathcal{G}}

\newcommand\GFolDistr{\GDistr_{\Foliation}}
\newcommand\clGBunchFol{\overline{\GFolDistr}}

\newcommand\Gy[1]{G_{#1}}
\newcommand\UIdProp{{\mycolor{blue}N}}

\newcommand\NbhId{{\mycolor{blue}\mathcal{N}}}

\newcommand\GLESNbhId{{\mycolor{blue}\mathcal{M}}}

\begin{document}
\begin{abstract}
Let $M$ be a smooth manifold and $\mathcal{F}$ a Morse-Bott foliation with a compact critical manifold $\Sigma\subset M$.
Denote by $\mathcal{D}(\mathcal{F})$ the group of diffeomorphisms of $M$ leaving invariant each leaf of $\mathcal{F}$.
Under certain assumptions on $\mathcal{F}$ it is shown that the computation of the homotopy type of $\mathcal{D}(\mathcal{F})$ reduces to three rather independent groups: the group of diffeomorphisms of $\Sigma$, the group of vector bundle automorphisms of some regular neighborhood of $\Sigma$, and the subgroup of $\mathcal{D}(\mathcal{F})$ consisting of diffeomorphisms fixed near $\Sigma$.
Examples of computations of homotopy types of groups $\mathcal{D}(\mathcal{F})$ for such foliations are also presented.
\end{abstract}

\maketitle

\section{Introduction}
The paper continues a series of works~\cite{KhokhliukMaksymenko:IndM:2020, KhokhliukMaksymenko:PIGC:2020, KhokhliukMaksymenko:Nbh:2022, KhokhliukMaksymenko:JHRS:2023, Maksymenko:Lens2:2022, Maksymenko:PocIMNANUS:2023} devoted to computations of the homotopy types of groups of leaf-preserving diffeomorphisms for certain classes of foliations with singularities including Morse-Bott ones.

The general problem of determining homotopy types of different kinds of spaces of maps and groups of automorphisms (especially for smooth maps between manifolds) is extremely hard, and explicit computations are made in not so many cases.
On the other hand, the importance of the information about homotopy types even to real world problems can be observed from the facts that usually real objects and processes are ``invariant with respect to small deformations'', and therefore their integral numeric characteristics (like degree of a map) should be ``stable under such deformations'' and thus represented by some homotopy invariants.
Let us just recall that the homotopy types of diffeomorphism or homeomorphism groups are known for all manifolds of dimension $\leqslant 2$, for many $3$-manifolds, and for certain specific manifolds of dimensions $\geqslant4$, see e.g.~\cite{Chernavski:MSb:1969, EarleEells:JGD:1969, Hamstrom:BAMS:1974, WeissBruce:AMS:2001, Smolentsev:SMP:2006, HongKalliongisMcCulloughRubinstein:LMN:2012} and references therein.
For groups of diffeomorphisms preserving some ``geometric'' structures there is essentially less information, see~\cite{FukuiUshiki:JMKU:1975, Fukui:JJM:1976, Maksymenko:TA:2003, Maksymenko:OsakaJM:2011} especially for computing homotopy types of diffeomorphism groups preserving several classes of foliations on $3$-manifolds and $1$-dimensional foliations with singularities (flow lines).

In~\cite{KhokhliukMaksymenko:Nbh:2022} it was shown that given a manifold $\Mman$, a submanifold $\fSing \subset \Mman$, and its regular neighborhood $\vbp:\Eman\to\fSing$, one can deform by isotopy each diffeomorphism $\dif$ of $\Mman$ preserving $\fSing$ to a diffeomorphism which coincides near $\fSing$ with some (uniquely determined) vector bundle automorphism $\hat{\dif}$ of $\Eman$ given by $\hat{\dif}(\pv)=\lim\limits_{t\to0}\frac{1}{t}\dif(t\pv)$.
Such a statement can be regarded mutually as a parametrized, compactly supported, and even foliated extension of a theorem on isotopies of regular neighborhoods.
It is convenient to mention this procedure as a ``linearization'' of $\dif$.
Note that the correspondence $\tFibMap:\dif\mapsto\hat{\dif}$ can be regarded as a homomorphism from the group $\DiffInv{\Mman}{\fSing}$ of diffeomorphisms of $\Mman$ preserving $\fSing$ into the group $\GLE$ of all vector bundle automorphisms of $\vbp$.

That linearization allowed to compute in~\cite{KhokhliukMaksymenko:JHRS:2023, Maksymenko:Lens2:2022, Maksymenko:PocIMNANUS:2023} the homotopy types of the group $\Diff(\Foliation)$ of leaf preserving diffeomorphisms for certain Morse-Bott foliations $\Foliation$ with only center singularities on the solid torus and all lens spaces as well as on their non-orientable counterparts: solid Klein bottle and twisted $S^2$ bundle over $\Circle$.
Such foliations and flows preserving them are studied in~\cite{ScarduaSeade:JDG:2009, ScarduaSeade:JT:2011, MafraScarduaSeade:JS:2014, HatamianPrishlyak:PIGC:2020}.

It is also proved in~\cite{KhokhliukMaksymenko:Nbh:2022} that $\tFibMap:\DiffInv{\Mman}{\fSing}\to\GLE$ is a locally trivial fibration \term{over its image}, see Theorem~\ref{th:isot_nbh}.
The main result of the present paper, see Theorems~\ref{th:KhM:Fol} and~\ref{th:main_res_mb}, shows that $\tFibMap$ is still a locally trivial fibration over its image if we replace $\DiffInv{\Mman}{\fSing}$ with the group $\Diff(\Foliation)$ of leaf preserving diffeomorphisms of a foliation $\Foliation$ from a certain class which includes Morse-Bott foliations \term{linearizable} near their singular submanifolds, see Definition~\ref{def:mf-fol-homog}.
Together with the above ``linearization'' technique this allows to reduce in certain cases the computation of the homotopy type of $\Diff(\Foliation)$ to three rather independent ingredients: (a)~the group $\Diff(\fSing)$ of diffeomorphisms of $\fSing$; (b)~the group $\GLV{\Foliation,\fSing}$ of vector bundle automorphisms of some regular neighborhood $\Eman$ of $\fSing$ mutually preserving the leaves of $\Foliation$ and fibers of $\Eman$; and (c)~the group $\DiffNb{\Foliation}{\fSing}$ of $\Foliation$-leaf preserving fixed on some neighborhood of the ``singular'' manifold $\fSing$, i.e.\ supported out of the singularity.

As an illustration of usefulness of the developed methods, we will present now a particular case of computations given in Section~\ref{sect:applications}, see Theorem~\ref{th:Omega_On_O2__DFol} and Corollary~\ref{cor:homtype_DF_DnS1_n23}, which reprove one of the main results from~\cite{KhokhliukMaksymenko:JHRS:2023}.

\subsection{Morse-Bott foliation on $D^n\times\Circle$ with a singular ``central'' circle}
For each $\rr\geqslant0$ let $\Cman_{\rr}\subset\bR^{n}$ be the $(n-1)$-sphere of radius $\rr$ centered at the origin $0\in\bR^{n}$, $D^{n}$ be the closed unit $n$-disk also with the center at the origin, and $\Circle$ be the unit circle in the complex plane.
Let also $\ATor=D^{n}\times\Circle$, and $\Foliation = \{ \Cman_{\rr}\times\Circle \mid \rr\in[0;1] \}$ be the partition of $\ATor$ into ``parallel copies'' of $\dATor = S^{n-1}\times\Circle$ and the central circle $0\times\Circle$.
One can regard $\Foliation$ as the partition of $\ATor$ into level sets of the following Morse-Bott function $\func:\ATor\to\bR$, $\func(\pv,\pbs) = \nrm{\pv}^2$.
Denote by $\Diff(\Foliation)$ the group of diffeomorphisms of $\ATor$ leaving invariant each leaf of $\Foliation$, by $\DiffFix{\Foliation}{\Xman}$ the subgroup of $\Diff(\Foliation)$ consisting of diffeomorphisms fixed on a subset $\Xman\subset\ATor$.

Further, let $\Omega^{\infty}(\SO(n)) = \Ci{(\Circle,1)}{(\SO(n),I)}$ be the space of smooth loops in the special linear group $\SO(n)$ at the unit matrix $I$.
Then we have a natural inclusion
\begin{gather*}
    \eta:\Omega^{\infty}(\SO(n))\times\Ort(n)\times\Ort(2) \to \Diff(\Foliation) \subset \Diff(\ATor), \\
    \eta(\Phi, \Alin, \Blin)(\pv,\pbs) = ( \Phi(y) \Alin \pv, \Blin\pbs ),
\end{gather*}
whose image is evidently contained in $\Diff(\Foliation)$.

Moreover, let $\rho:\Diff(\ATor)\to\Diff(\dATor)$, $\rho(\dif)=\restr{\dif}{\dATor}$, be the ``restriction to the boundary'' homomorphism being (due to J.~Cerf) a locally trivial fibration over its image, see Theorem~\ref{th:DinvMS_split} below.
In particular, that image is a union of path components of $\Diff(\dATor)$.

A.~Hatcher~\cite{Hatcher:ProcAMS:1981} shown that for $n=3$ the composition
\[
    \rho\circ\eta:\Omega^{\infty}(\SO(3))\times\Ort(3)\times\Ort(2)\to\Diff(\dATor)
\]
is a homotopy equivalence.

In the case $n=2$ one can say more.
Namely, let $\hLambda, \hMu, \hTau, \hDtwist, \rotat{a}{b} \in\Diff(\dATor)$, $(a,b)\in\Circle\times\Circle$, be diffeomorphisms of $\dATor=\Circle\times\Circle$ given by
\begin{align}\label{eqt:LMTD}
 &\hLambda(\pz,\pbs) = (\pz, \bar{\pbs}),
&&\hMu(\pz,\pbs)     = (\bar{\pz}, \pbs),
&&\hDtwist(\pz,\pbs) = (\pz\pbs, \pbs),
&&\rotat{a}{b}(\pz,\pbs)=(a\pz, b\pbs),
\end{align}
and $\calQ = \gen{\hDtwist, \hLambda, \hMu, \rotat{a}{b}, a,b\in\Circle}$ be the subgroup of $\Diff(\Foliation)$ generated by them.
Evidently, each of these diffeomorphisms extends to a diffeomorphism of all $\ATor$, so we have the following maps:
\begin{equation}\label{equ:incl_Omega_O2_O2__DFol}
    \Omega^{\infty}(\SO(2))\times\Ort(2)\times\Ort(2)
    \xrightarrow{~\eta~}
    \Diff(\ATor)
    \xrightarrow{~\rho~}
    \rho\bigl(\Diff(\ATor)\bigr)
    \supset
    \calQ.
\end{equation}
Then it is a classical result that all these maps and inclusions are homotopy equivalences, e.g.~\cite{Ivanov:LOMI:1976, Wajnryb:FM:1998}, see also~\cite[Section~6]{KhokhliukMaksymenko:JHRS:2023} for discussion.
In particular, those spaces have countably many path components each of which has the homotopy type of $\Circle\times\Circle$.

The following statement is a particular case of Corollary~\ref{cor:th:homtype_DFol_extr_mb-fol} and is contained in~\eqref{equ:hom_eq:DFCircle}.
\begin{subtheorem}\label{th:Omega_On_O2__DFol}
The map $\eta:\Omega^{\infty}(\SO(n))\times\Ort(n)\times\Ort(2) \to \Diff(\Foliation)$ is a homotopy equivalence for each $n\geqslant2$.
\end{subtheorem}

\begin{subcorollary}[{\rm c.f.~\cite[Theorem~1.2.1]{KhokhliukMaksymenko:JHRS:2023}}]
\label{cor:homtype_DF_DnS1_n23}
{\rm 1)}~For $n=2$ all the arrows and inclusions are homotopy equivalences:
\[
    \Omega^{\infty}(\SO(2))\times\Ort(2)\times\Ort(2)
    \xrightarrow{~\eta~}
    \Diff(\Foliation)
    \ \subset \
    \Diff(\ATor)
    \xrightarrow{~\rho~}
    \rho\bigl(\Diff(\ATor)\bigr)
    \ \supset \
    \calQ.
\]

{\rm~2)}~For $n=3$ both the arrows:
\[
    \Omega^{\infty}(\SO(3))\times\Ort(3)\times\Ort(2)
    \xrightarrow{~\eta~}
    \Diff(\Foliation)
    \xrightarrow{~\rho~}
    \Diff(\dATor)
\]
are also homotopy equivalences.

{\rm~3)}~Moreover, for both $n=2,3$ the group $\DiffFix{\Foliation}{\dATor}$ is weakly contractible (i.e.\ all its homotopy groups vanish).
\end{subcorollary}
\begin{proof}
Statements 1) and 2) follow from discussions above and Theorem~\ref{th:Omega_On_O2__DFol}.

3) By~\cite[Theorem~4.3]{KhokhliukMaksymenko:IndM:2020} or \cite[Lemma~3.2.2]{KhokhliukMaksymenko:JHRS:2023}, for all $n$ the map $\restr{\rho}{\Diff(\Foliation)}\colon\Diff(\Foliation)\to\Diff(\dATor)$ to the subgroup $\Diff(\Foliation)$ is also a locally trivial fibration over its image with fiber $\DiffFix{\Foliation}{\dATor}$.
Moreover, each 1) and 2) imply that $\rho\colon\Diff(\Foliation)\to\rho(\Diff(\Foliation))$ is a homotopy equivalence, whence the homotopy groups of the fiber $\DiffFix{\Foliation}{\fSing}$ should vanish.
\end{proof}

Corollary~\ref{cor:homtype_DF_DnS1_n23} is proved in~\cite[Theorem~1.1.2]{KhokhliukMaksymenko:JHRS:2023} for $n=2$ in ``opposite order'': it first establishes that $\DiffFix{\Foliation}{\fSing}$ is weakly contractible via rather geometric methods, and then uses the fibration $\restr{\rho}{\Diff(\Foliation)}\colon\Diff(\Foliation)\to\rho(\Diff(\dATor))$ to show that it is a homotopy equivalence.

\section{Preliminary definitions and statements}\label{sect:preliminaries}
In what follows all manifolds are assumed to be smooth ($\Cinfty$), they may have boundary, and their connected components may have distinct dimensions.
For $n\geqslant1$ we will denote by $\MatR{n}$ the space of square $(n\times n)$-matrices which we identify with $\bR^{n^2}$ by regarding the entries of those matrices as coordinates.
The arrows $\monoArrow$ and $\epiArrow$ are always used to indicate \term{monomorphism} or \term{epimorphism} respectively, while the signs $\cong$ and $\simeq$ mean respectively \term{homeomorphism} and \term{homotopy equivalence}.

\subsection{Local sections}
For a topological group $G$ denote by $e_{G}$ its unit element and by $G_{0}$ the path component of $G$ containing $e_{G}$.
Then $G_{0}$ is a normal subgroup of $G$, and the quotient $G/G_{0}$ is canonically identified with the set $\pi_0 G$ of path components of $G$.

Let $\phi\colon G \to H$ be a continuous homomorphism to another topological group $H$.
By a \term{local section} of $\phi$ we will mean a continuous map $s\colon U \to G$ defined on some open neighborhood $U$ of the unit $e_{H}$ of $H$ such that $\phi\circ s = \id_{U}$.
If $U=H$, then $s$ is called a \term{global} section.
The following lemma is well known.
\begin{sublemma}[{\rm e.g.~\cite[Theorem~A]{Palais:CMH:1960}}]\label{lm:principal_fibrations}
Let $\phi\colon G \to H$ be a continuous homomorphism between topological groups.
Denote $K=\ker(\phi)$ and $K'=K\cap G_0$.
Suppose $\phi$ admits a local section, i.e.\ a continuous map $s\colon U \to G$ defined on some open neighborhood $U$ of the unit $e_{H}$ of $H$ such that $\phi\circ s = \id_{U}$.
Then $\phi\colon  G \to \phi(G)$ is a locally trivial principal $K$-fibration over the image $\phi(G)$.
If, in addition, $\phi(G)$ is paracompact and Hausdorff, which holds e.g.\ if $H$ is metrizable.
Then the following statements hold.
\begin{enumerate}[label={\rm(\arabic*)}, leftmargin=*]
\item
The restriction $\phi:G_{0}\to H_{0}$ is surjective and is a principal $K'$-fibration.
\item
$\phi(G)$ is a union of path components of $H$, $\phi\colon  G \to \phi(G)$ satisfies the homotopy lifting axiom, and we have the following long exact sequence of homotopy groups:
\begin{align*}
\cdots \to \pi_k(K_0, \unit{G}) \to \pi_k(G_0, \unit{G}) &\to \pi_k(H_0, \unit{H}) \to \pi_{k-1}(K_0, \unit{G})  \to \cdots \\
\cdots
&\to \pi_1(H_0, \unit{H}) \to K/K_0 \to G/G_0 \to G/K.
\end{align*}
\item\label{enum:lm:principal_fibrations:global_sect}
If $U=H$, i.e.\ the section $s$ is global, then there is a homeomorphism $\xi:K\times H \to G$, $\xi(k,h) = k \cdot s(h)$.
If, in this case, $K$ is contractible, then $\phi$ is a homotopy equivalence.
Moreover, if, in addition, $s$ is a \term{homomorphism}, then $\xi$ is an isomorphism of the semi-direct product $K\rtimes H$, corresponding to the action of $s(H)$ on $K$ by conjugations, onto $G$.
More precisely, $K\rtimes H$ is a direct product of sets $K\times H$ with the following operation: $(k_1,h_1)(k_2,h_2)=(k_1 s(h_1) k_2 s(h_1)^{-1}, h_1h_2)$, and we have the following commutative diagram:
\[
    \xymatrix@C=8em@R=2em{
        K \ar@{^(->}[r]^-{k \mapsto (k, e_{H})} \ar@{=}[d] &
        K\rtimes H  \ar@{->>}[r]^-{(k,h) \mapsto h} \ar[d]_-{\xi: (k,h) \mapsto k s(h)} &
        H \ar@{=}[d] \\
        K \ar@{^(->}[r] & G \ar[r]^-{\phi} & H
    }
\]
\end{enumerate}
\end{sublemma}

\subsection{Diffeomorphisms preserving a submanifold}
Let $\Mman$ be a compact manifold.
Denote by $\Diff(\Mman)$ the group of $\Cinfty$ diffeomorphisms of $\Mman$ endowed with the strong $\Cinfty$ Whitney topology.
Note that this topology is generated by a certain metric.
For a subset $\fSing\subset\Mman$ define $\DiffInv{\Mman}{\fSing}$, $\DiffFix{\Mman}{\fSing}$, $\DiffNb{\Mman}{\fSing}$ to be the subgroups of $\Diff(\Mman)$ consisting of diffeomorphisms respectively leaving $\fSing$ \term{invariant}, \term{fixed} on $\fSing$, and fixed on some \term{neighborhood} of $\fSing$, i.e.\ supported in $\Mman\setminus\fSing$.

A submanifold $\fSing$ of $\Mman$ will be called \term{proper}%
\footnote{Usually the term \term{proper} is used only when $\partial\fSing = \Mman \pitchfork \fSing$, but in the present paper it will be convenient to include also the case when some components of $\fSing$ coincide with boundary components of $\Mman$.}
if every connected component $\fSing'$ of $\fSing$ is either a boundary component of $\Mman$ or $\partial\fSing' = \Mman \cap \fSing'$ and this intersection is transversal.

The following well-known result is independently proved by J.~Cerf~\cite[\S2.2.1, Theorem~5]{Cerf:BSMF:1961}, R.~Palais~\cite[Theorem~B]{Palais:CMH:1960}, and E.~Lima~\cite{Lima:CMH:1964}:
\begin{subtheorem}[\cite{Cerf:BSMF:1961, Palais:CMH:1960, Lima:CMH:1964}]
\label{th:DinvMS_split}
Let $\fSing$ be a proper compact submanifold.
Then the natural ``\term{restriction to $\fSing$}'' homomorphism
\[
    \tRestr\colon\DiffInv{\Mman}{\fSing}\to\Diff(\fSing),
    \qquad
    \tRestr(\dif) = \restr{\dif}{\fSing},
\]
with kernel $\DiffFix{\Mman}{\fSing}$, admits a local section.
\end{subtheorem}
\begin{subremark}\rm
In fact, Theorem~\ref{th:DinvMS_split} is proved in a more general context: namely the same ``restriction to $\fSing$'' map $\Diff(\Mman) \to \Emb(\fSing,\Mman)$, $\dif \mapsto \dif|_{\fSing}$, into the space of all embeddings of $\fSing$ into $\Mman$ admits a local section $s$ defined on some open neighborhood $\UU$ of the identity embedding $\fSing\subset\Mman$.
It then follows that the restriction of $s$ to $\UU\cap\Diff(\fSing)$ gives a local section from Theorem~\ref{th:DinvMS_split}.
\end{subremark}

\begin{subremark}\rm
By Lemma~\ref{lm:principal_fibrations} existence of local section implies that $\tRestr$ is a principal locally trivial $\DiffFix{\Mman}{\fSing}$-fibration over its image being the union of path components of $\Diff(\fSing)$, the map $\tRestr$ satisfies the path lifting axiom which can be interpreted as the Isotopy Extension Theorem, there is a long exact sequence of homotopy groups of that fibration, etc.
Note that even in the ``simplest'' case when $(\Mman,\fSing) = (D^{n}, S^{n-1})$, and thus $\DiffInv{D^n}{S^{n-1}}=\Diff(D^n)$, the map $\tRestr$ plays an important role.
Namely, the quotient $\Gamma^n:=\pi_0\Diff(S^{n-1})/\tRestr\bigl(\pi_0\Diff(D^{n})\bigr)$ of the image of that restriction homomorphism on the level of groups of path-components is an abelian group appearing in the classification smooth structures on manifolds, and it also coincides with the group $\Theta_n$ of homotopy spheres for $n>5$, see e.g.~\cite{Thom:ProcICM:1958, Munkres:AnnMath:1960, KervaireMilnor:AnnMAth:1963, Cerf:LMN:1968}, and also~\cite{HendriksMcCullough:TA:1987} for other applications of that theorem.
\end{subremark}

Further, let $\vbp\colon\Eman\to\fSing$ be a \term{regular neighborhood} of $\fSing$ in $\Mman$, i.e.\ $\vbp$ is a smooth retraction defined on some open neighborhood $\Eman$ of $\fSing$ in $\Mman$ with a fixed vector bundle structure for $\vbp$.
Since we allow that $\dim\fSing' \neq \dim\fSing''$ for distinct components $\fSing'$ and $\fSing''$ of $\fSing$, the dimensions of the fibers over $\fSing'$ and $\fSing''$ may also be distinct.
We will also regard $\fSing$ as the image of the zero section in $\Eman$.

Denote by $\GLE$ the group of $\Cinfty$ vector bundle automorphisms of $\Eman$ (not necessarily fixed on $\fSing$ and thus allowed to interchange the fibers of $\vbp$), and by $\GLES$ its subgroup consisting of automorphisms fixed on $\fSing$, i.e.\ leaving invariant each fiber of $\vbp$.

Let also $\Uman \subset \Eman$ be an open neighborhood of $\fSing$, and $\dif\colon\Uman\to\Eman$ be a $C^{\infty}$ embedding such that $\dif(\fSing) = \fSing$.
Then $\dif$ yields the following $\Cinfty$ vector bundle automorphism:
\begin{equation}\label{equ:Tfibh}
    \tfib{\dif}\colon\Eman\to\Eman,
    \qquad
    \tfib{\dif}(\pv) = \lim_{t\to 0}\tfrac{1}{t} \dif(t\pv), \ \pv\in\Uman.
\end{equation}
It can be regarded as the \term{tangent map of $\dif$ along fibers of $\vbp$} at points of $\fSing$, see~\cite{KhokhliukMaksymenko:Nbh:2022} for a detailed discussions.
For instance, if $\fSing$ is a point, then $\Eman = \bR^{n}$ for some $n\geqslant0$, and $\tfib{\dif}\colon\bR^{n}\to\bR^{n}$ is the (linear) tangent map of $\dif$ at $\fSing$ given by the Jacobi matrix of $\dif$.

Notice that the correspondence
\begin{equation}\label{equ:map_tanh}
    \tFibMap\colon \DiffInv{\Mman}{\fSing} \to \GLE,
    \qquad
    \tFibMap(\dif) = \tfib{\dif},
\end{equation}
is a well-defined continuous homomorphism.
Let $\DiffFix{\Mman}{\fSing, 1}:=\ker(\tFibMap)$ be its kernel, so
\begin{equation}\label{equ:ker_tanh}
\begin{aligned}
    \DiffFix{\Mman}{\fSing, 1} &=
    \{ \dif\in \DiffFix{\Mman}{\fSing} \mid \tfib{\dif} = \id_{\Eman} \} = \\
    &=\{ \dif\in \DiffFix{\Mman}{\fSing} \mid
         \tang[\pbs](\dif)=\id_{\tang[\pbs]{\Eman}} \colon \tang[\pbs]{\Eman} \to \tang[\pbs]{\Eman}
         \ \text{for all} \
         \pbs\in\fSing
      \},
\end{aligned}
\end{equation}
consists of diffeomorphisms fixed on $\fSing$ and whose tangent map along fibers $\tfib{\dif}$ is the identity.
Since the tangent space $\tang[\pbs]{\Eman}$ splits into the direct sum $\tang[\pbs]{\fSing} \oplus \tang[\pbs]{\Eman_{\pbs}}$ of the tangent spaces to $\fSing$ and to the fiber $\Eman_{\pbs}:=\vbp^{-1}(\pbs)$, we get another equivalent description of the kernel of $\tFibMap$ (indicated at the second line of~\eqref{equ:ker_tanh}): it consists of diffeomorphisms fixed on $\fSing$ and whose tangent map is the identity at each $\pbs\in\fSing$.

Let also
\begin{equation}\label{equ:Diff_InvFix_MSp}
\begin{aligned}
    \DiffInv[\vbp]{\Mman}{\fSing} &:=
    \bigl\{
        \dif\in \DiffInv{\Mman}{\fSing} \mid
        \dif = \tfib{\dif} \ \text{near} \ \fSing
    \bigr\}, \\
    \DiffFix[\vbp]{\Mman}{\fSing} &:= \DiffInv[\vbp]{\Mman}{\fSing} \cap \DiffFix{\Mman}{\fSing},
\end{aligned}
\end{equation}
be the subgroups of $\DiffInv{\Mman}{\fSing}$ consisting of diffeomorphisms leaving $\fSing$ invariant (resp.\ fixed on $\fSing$) and coinciding near $\fSing$ with some automorphism of $\Eman$.
Evidently,
\[
    \DiffNb{\Mman}{\fSing} = \DiffFix[\vbp]{\Mman}{\fSing} \cap \DiffFix{\Mman}{\fSing, 1}.
\]
Finally, if $\fSing\subset\Xman$ for some subset $\Xman\subset\Mman$, then it will also be convenient to put:
\begin{align*}
    \DiffFix[\vbp]{\Mman}{\Xman} &:= \DiffFix[\vbp]{\Mman}{\fSing}\cap \DiffFix{\Mman}{\Xman},\\
    \DiffFix[1]{\Mman}{\Xman}    &:= \DiffFix[1]{\Mman}{\fSing}\cap \DiffFix{\Mman}{\Xman},
\end{align*}
Let us collect all the above groups in a single diagram:
\begin{equation}\label{equ:diagam_homot_eq_all_maps}
\begin{aligned}
\xymatrix@C=0.1em@R=0.7em{
 & & & & & \fbox{$\GLES$} \ar@{^(->}[dddd]  \\
 & \DiffFix[\vbp]{\Mman}{\fSing}           \ar@{^(->}[dd] \ar@/^9pt/@{->}[rrrru]^-{\tFibMap} \ar@{^(->}[rr]^-{\simeq} &&
   \DiffFix{\Mman}{\fSing}                 \ar@{^(->}[dd] \ar@{-->}[rru]^-{\tFibMap}     \\
   \fbox{$\DiffNb{\Mman}{\fSing}$}         \ar@/^5pt/@{^(->}[ru] \ar@/_5pt/@{^(->}[rd] \ar@{^(-->}[rr]^(0.6){\simeq}  &&
   \DiffFix{\Mman}{\fSing, 1}              \ar@{^(-->}[ru] \ar@{^(-->}[rd]    \\
 & \DiffInv[\vbp]{\Mman}{\fSing}           \ar[dd]^-{\tRestr} \ar@/_9pt/@{->}[rrrrd]^-{\tFibMap} \ar@{^(->}[rr]^-{\simeq} &&
   \fbox{\fbox{$\DiffInv{\Mman}{\fSing}$}} \ar[dd]^(.7){\tRestr} \ar@{-->}[rrd]^-{\tFibMap}    \\
 & & & & & \GLE \ar[d]^-{\tRestr}\\
 & \Diff(\fSing) \ar@{=}[rr] && \fbox{$\Diff(\fSing)$} \ar@{=}[rr] && \Diff(\fSing)
}
\end{aligned}
\end{equation}
where the right lower arrow $\tRestr\colon\GLE\to\Diff(\fSing)$ also associates to each $\Alin\in\GLE$ its restriction to $\fSing$.

\begin{subdefinition}\label{def:repr_fibr}
Say that a pair of continuous homomorphisms of topological groups $\calP\xmonoArrow{\alpha} \calQ \xrightarrow{\,\beta\,} \calR$  \term{constitute a fibration over its image}, if $\alpha$ is an inclusion, $\beta$ is a homomorphism with kernel $\calP$, and it admits a local section $s:\UU\to\calQ$ defined on some open neighborhood $\UU$ of the unit of $\calR$.
In this case, see Lemma~\ref{lm:principal_fibrations}, $\beta$ is indeed a locally trivial principal $\calP$-fibration over its image $\beta(\calP)$ being a union of path components of $\calR$ etc.
Note however that $\beta$ is not necessarily surjective.
If $\beta$ is surjective, then $\calP\xmonoArrow{\alpha} \calQ \xepiArrow{\,\beta\,} \calR$ will be called a \term{fibration}.
\end{subdefinition}

\begin{subtheorem}[\cite{KhokhliukMaksymenko:Nbh:2022}]
\label{th:isot_nbh}
Let $\Mman$ be a manifold, $\fSing \subset\Mman$ a proper compact submanifold, and $\vbp\colon\Eman\to\fSing$ a regular neighborhood of $\fSing$ in $\Mman$.
Then the following statements hold.
\begin{enumerate}[label={\rm(\arabic*)}, leftmargin=*]
\item\label{th:isot_nbh:linearization}
The inclusion of triples, see horizontal arrows $\xmonoArrow{~\simeq~}$ in~\eqref{equ:diagam_homot_eq_all_maps},
\begin{equation}\label{equ:DMS_homot_equiv}
\begin{array}{lcr}
    \bigl(
        \DiffInv[\vbp]{\Mman}{\fSing}, \ & \
        \DiffFix[\vbp]{\Mman}{\fSing}, \ & \
        \DiffNb{\Mman}{\fSing}
    \bigr)  \\
    & \cap & \\
    \bigl(
        \DiffInv{\Mman}{\fSing}, &
        \DiffFix{\Mman}{\fSing}, &
        \DiffFix{\Mman}{\fSing, 1}
    \bigr)
\end{array}
\end{equation}
is a homotopy equivalence.

\item\label{th:isot_nbh:DfMSp_split}
Every consecutive pair of vertical, diagonal or dashed arrows $\calP \xmonoArrow{} \calQ \xrightarrow{\,\beta\,} \calR$ from the diagram~\eqref{equ:diagam_homot_eq_all_maps} constitute a fibration over its image.
\end{enumerate}
\end{subtheorem}
In fact, a more detailed relative variant holds:
\begin{subaddendum}\label{add:main_th:X}
Let $\Xman\subset\Mman\setminus\fSing$ be a closed (possibly empty) subset.
Then there exists a deformation
\[ \Hhmt\colon\DiffInv{\Mman}{\fSing}\times[0;1]\to\DiffInv{\Mman}{\fSing} \]
of the lower triple of~\eqref{equ:DMS_homot_equiv} into the upper triple such that
\begin{enumerate}[label={\rm(\arabic*)}]
\item 
$\tFibMap(\Hhmt(\dif,t))=\tFibMap(\dif)$ for all $\dif\in\DiffInv{\Mman}{\fSing}$ and $t\in[0;1]$;
\item
$\restr{\Hhmt(\dif,t)}{\Xman}=\restr{\dif}{\Xman}$ for all $\dif\in\DiffInv{\Mman}{\fSing}$ and $t\in[0;1]$;
\item
except for the lower right arrow $\tRestr\colon\GLE\to\Diff(\fSing)$ for each $\dif\in\UU$ its image $s(\dif)$ under the local section $s$ is fixed on $\Xman$.
\end{enumerate}
In other words, Theorem~\ref{th:isot_nbh} holds if we replace each group $\Diff_{*}(\Mman,-)$ with the intersection $\Diff_{*}(\Mman,-)\cap\DiffFix{\Mman}{\Xman}$.
\end{subaddendum}

Theorem~\ref{th:isot_nbh} can also be stated so that~\eqref{equ:diagam_homot_eq_all_maps} consists of two $(3\times3)$-diagrams:
\begin{align*}
    &\xymatrix@C=1.2em@R=1.5em{
        \DiffNb{\Mman}{\fSing}  \ar@{=}[d]      \ar@{^(->}[r] &
        \DiffFix[\vbp]{\Mman}{\fSing} \ar@{^(->}[d]   \ar[r]^-{\tFibMap} &
        \GLES                   \ar@{^(->}[d]  \\
        \DiffNb{\Mman}{\fSing}  \ar[d]   \ar@{^(->}[r] &
        \DiffInv[\vbp]{\Mman}{\fSing} \ar[d]^-{\tRestr}   \ar[r]^-{\tFibMap} &
        \GLE                    \ar[d]^-{\tRestr}    \\
        1  \ar@{^(->}[r] &
        \Diff(\fSing) \ar@{=}[r] &
        \Diff(\fSing)
    }
    &&
    \xymatrix@C=1.2em@R=1.5em{
        \DiffFix[1]{\Mman}{\fSing}  \ar@{=}[d]      \ar@{^(->}[r] &
        \DiffFix{\Mman}{\fSing} \ar@{^(->}[d]   \ar[r]^-{\tFibMap} &
        \GLES                   \ar@{^(->}[d]  \\
        \DiffFix[1]{\Mman}{\fSing}  \ar[d]   \ar@{^(->}[r] &
        \DiffInv{\Mman}{\fSing} \ar[d]^-{\tRestr}   \ar[r]^-{\tFibMap} &
        \GLE                    \ar[d]^-{\tRestr}    \\
        1  \ar@{^(->}[r] &
        \Diff(\fSing) \ar@{=}[r] &
        \Diff(\fSing)
    }
\end{align*}
in which all rows and columns represent fibrations, and the right diagram ``can be deformed'' into the left diagram with fixed lower row and right column.

The proof is discussed in Section~\ref{sect:rem:proof:th:isot_nbh}.
Theorem~\ref{th:isot_nbh} implies that, due to long exact sequences of homotopy groups of the corresponding fibrations, the study of the homotopy type of $\DiffInv{\Mman}{\fSing}$ might be reduced to the study of the homotopy type of $\DiffInv[\vbp]{\Mman}{\fSing}$ which in turn splits into three (in a certain sense independent) parts included in~\eqref{equ:diagam_homot_eq_all_maps} into rectangular boxes.
\begin{itemize}[leftmargin=*]
\item
\term{The group $\Diff(\fSing)$ of diffeomorphisms of the manifold $\fSing$}.
As mentioned in the introduction, if $\dim\fSing\leqslant 2$ then the homotopy types of $\DiffId(\fSing)$ are completely known, \cite{EarleEells:JGD:1969, EarleSchatz:DG:1970, Gramain:ASENS:1973}, while in dimension $3$ we have a lot of information on them, e.g.~\cite{Hatcher:AnnM:1983, Hatcher:ProcAMS:1981, Gabai:JDG:2001, HongKalliongisMcCulloughRubinstein:LMN:2012};

\item
\term{The group $\GLES$} of vector bundle automorphisms of $\Eman$ fixed on $\fSing$.
It can be identified with the spaces of $\Cinfty$ sections of a certain principal $\GL(\bR^n)$-fibration over $\fSing$, see e.g.~\cite[Ch.~1.2]{Atiyah:Kth:1967}, and therefore can be studied by purely homotopical methods;

\item
\term{The group $\DiffNb{\Mman}{\fSing}$ of diffeomorphisms fixed near $\fSing$}, i.e.\ supported out of $\fSing$.
Removing from $\Mman$ a thin neighborhood $\ATor$ of $\fSing$, one get a manifold $\Mman_1:=\overline{\Mman\setminus\ATor}$ with additional boundary component $\dATor$ being a spherical bundle over $\fSing$, one can identify $\DiffNb{\Mman}{\fSing}$ with the group $\DiffNb{\Mman_1}{\dATor}$ of diffeomorphisms of $\Mman_1$ fixed near $\dATor$.
\end{itemize}

Our principal goal is to extend the above scheme to diffeomorphisms preserving leaves of certain singular foliations, and this was partially done in previous papers.
Namely, an analogue of Theorem~\ref{th:DinvMS_split} for leaf preserving diffeomorphisms of Morse-Bott foliations is proved in~\cite{KhokhliukMaksymenko:IndM:2020}, while an analogue of Theorem~\ref{th:isot_nbh}\ref{th:isot_nbh:linearization} for \term{``scalable'' near $\fSing$} foliations is also established in~\cite{KhokhliukMaksymenko:Nbh:2022}.

The main result of the present paper is Theorem~\ref{th:KhM:Fol} being a counterpart of Theorem~\ref{th:isot_nbh} for leaf preserving diffeomorphisms of rather large classes of partitions of $\Mman$ with the mentioned above scalability properties and certain additional assumptions near $\fSing$.
In particular, Theorem~\ref{th:KhM:Fol} holds for a special class of Morse-Bott foliations, see Theorem~\ref{th:main_res_mb}.

\subsection{Partitions with scalability properties}
Let $\Foliation$ be a partition of a manifold $\Mman$.
The elements of the partition $\Foliation$ of $\Mman$ will be called \term{leaves}.
A subset $\Aman\subset\Mman$ is \term{$\Foliation$-saturated}, whenever $\Aman$ is a union of leaves of $\Foliation$.
In this case by the \term{restriction of $\Foliation$ onto $\Aman$} we will mean a partition $\restr{\Foliation}{\Aman}$ of $\Aman$ into the leaves of $\Foliation$.

\begin{subdefinition}\label{def:homogeneous_foliation}
Suppose $\fSing$ is an $\Foliation$-saturated submanifold of $\Mman$ and $\vbp\colon\Eman\to\fSing$ is a regular neighborhood of $\fSing$.
Say that an open neighborhood $\Uman$ of $\fSing$ is
\begin{itemize}[leftmargin=*, itemsep=1ex]
\item
\term{$\Foliation$-isolating} (for $\fSing$) if \ $\Uman\cap (\overline{\omega}\setminus\omega)  \subset \fSing$ \ for each $\omega\in\Foliation$;

\item
\term{$\Foliation$-scalable} (with respect to $\vbp$) if $\Uman \subset \Eman$ and the following condition holds:
\begin{itemize}[label={$-$}, leftmargin=*, topsep=1ex, parsep=1ex]
\item
let $\px,\pbs\in\Uman$ be any points belonging to the same leaf of $\Foliation$ and $\tau>0$ be any number such that $\tau\px,\tau\pbs\in\Uman$;
then $\tau\px,\tau\pbs$ also belong to the same (but possibly another) leaf of $\Foliation$;
\end{itemize}

\item
\term{star-convex} if $\Uman\subset\Eman$ and $t\Uman\subset\Uman$ for all $t\in[0;1]$.
\end{itemize}
\end{subdefinition}

The following lemma gives an important example of such  neighborhoods.
\begin{sublemma}[{\cite{KhokhliukMaksymenko:Nbh:2022}}]
\label{lm:example_Fgood}
Let $\vbp\colon\Eman\to\fSing$ be a vector bundle over a smooth manifold $\fSing$ (not necessarily compact) and $\gfunc\colon\Eman\to\bR$ be a continuous function such that
\begin{enumerate}[leftmargin=5ex]
\item\label{enum:vbhomog:cond:homog}
$\gfunc$ is homogeneous of some (not necessarily integer) degree $k>0$ on fibers, that is $\gfunc(\tau\pv) = \tau^{k}\gfunc(\pv)$ for all $\tau\geqslant0$ and $\pv\in\Eman$;
\item\label{enum:vbhomog:cond:pathcomp}
for some $a<0$ and $b>0$ the path components of $\gfunc^{-1}(a)$ and $\gfunc^{-1}(b)$ are closed in $\Eman$.
\end{enumerate}
Let also $\Foliation_{\gfunc}$ be the partition of $\Eman$ whose elements are path components of $\fSing$ and path components of the sets $\gfunc^{-1}(c)\setminus\fSing$ for all $c\in\bR$.
Then $\overline{\omega}\setminus\omega \subset\fSing$ for all $\omega\in\Foliation_{\gfunc}$ and every neighborhood of $\fSing$ is mutually \NGOOD{$\Foliation_{\gfunc}$}\ for $\fSing$.
\end{sublemma}

Thus, if $\vbp\colon\Eman\to\fSing$ is a regular neighborhood of $\fSing$, and there exists a function $\gfunc\colon\Eman\to\bR$ satisfying assumptions of Lemma~\ref{lm:example_Fgood} and such that $\Foliation|_{\Uman} = \Foliation_{\gfunc}|_{\Uman}$ on some open neighborhood $\Uman \subset\Eman$ of $\fSing$, then $\Uman$ is \NGOOD{$\Foliation_{\gfunc}$}\ for $\fSing$.

\subsection{Distributions of automorphism groups}\label{sect:dist_aut_groups}
Let $\vbp\colon\Eman\to\fSing$ be a vector bundle over a compact manifold $\fSing$.
Denote by $\EndES$ the $\bR$-algebra of all $\Cinfty$ vector bundle morphisms $\dif\colon\Eman\to\Eman$ fixed on $\fSing$, i.e.\ satisfying $\vbp\circ\dif=\vbp$.
Then $\GLES \subset \EndES$ is the multiplicative subgroup consisting of all vector bundle automorphisms.
Notice that the vector bundle projection $\vbp\colon\Eman\to\fSing$ is the zero element of the algebra $\EndES$, while $\id_{\Eman}$ is the unit of $\EndES$.

Evidently, each $\dif\in\EndES$ is uniquely determined by its restriction to any open neighborhood of $\fSing$.
In particular, if we fix some smooth compact tubular neighborhood $\Uman\subset\Eman$ of $\fSing$, then one can regard $\EndES$ as a subset of $\Ci{\Uman}{\Eman}$, and therefore we can endow $\EndES$ with the induced strong $\Cinfty$ Whitney topology.
It is easy to see that this topology does not depend on a particular choice of $\Uman$ and turns $\EndES$ into a Fr\'echet space.
Then $\GLES$ is an open subset of $\EndES$, and in particular, it is locally contractible.

For each $\pbs\in\fSing$ let $\Aut(\Eman_{\pbs})$ be the group of linear automorphisms of the fiber $\Eman_{\pbs}:=\vbp^{-1}(\pbs)$.
If $\dim\Eman_{\pbs}=n$, then every local trivialization chart $\phi\colon\vbp^{-1}(\Uman) \to \Uman \times \bR^{n}$, where $\Uman$ is a neighborhood of $\pbs$ in $\fSing$, gives a bijection $\phi_{\pbs}\colon\Aut(\Eman_{\pbs}) \to \GLR{n}$.
Hence, we can endow $\Aut(\Eman_{\pbs})$ with a Lie group structure such that $\phi_{\pbs}$ is an isomorphism of Lie groups, and it is easy to see that this structure does not depend on a choice of a local chart at $\pbs$.

\begin{subdefinition}
By a \term{distribution of automorphism groups over a subset $\Yman \subset \fSing$} or simply a \term{distribution over $\Yman$} we will mean an arbitrary collection of subgroups $\GDistr = \{ \Gy{\pbs} \subset \Aut(\Eman_{\pbs}) \}_{\pbs\in\Yman}$
of the corresponding groups of automorphisms of fibers over the points of $\Yman$.
The set $\Yman$ will be called the \term{support} of $\GDistr$.
Also, $\GDistr$ will be called a \term{Lie groups distribution} if each $\Gy{\pbs}$ is closed in $\Aut(\Eman_{\pbs})$, and therefore is a Lie subgroup of $\Aut(\Eman_{\pbs})$, e.g.\ \cite[Theorem~20.12]{Lee:Manifolds:2013}.
\end{subdefinition}
Thus, for each $\pbs\in\Yman$ one chooses some subgroup $\Gy{\pbs}$ of $\Aut(\Eman_{\pbs})$, and does not require any kind of continuity of such family.
Notice that then the distribution $\overline{\GDistr} = \{ \overline{\Gy{\pbs}} \}_{\pbs\in\Yman}$ consisting of closures of groups $\Gy{\pbs}$ is always a Lie groups distribution.
Also, one can always extend $\GDistr$ to have the support $\fSing$ by setting $\Gy{\pbs} := \Aut(\Eman_{\pbs})$ for $\pbs\in\fSing\setminus\Yman$.

\begin{sublemma}\label{lm:pres_func}
Let $\func\colon\Eman\to\bR$ be a continuous function, and for each $\pbs\in\fSing$ let
\[
    \Gy{\pbs} = \{ A \in \Aut(\Eman_{\pbs}) \mid \func(A(\pv)) = \func(\pv), \ \pv\in\Eman_{\pbs} \}
\]
be the group of linear automorphisms preserving the restriction of $\func$ on $\Eman_{\pbs}$.
Then each $\Gy{\pbs}$ is closed, so $\GDistr = \{\Gy{\pbs}\}_{\pbs\in\fSing}$ is a \term{Lie groups distribution over all of $\fSing$}.
\end{sublemma}
\begin{proof}
Evidently, we have a continuous action $\mu\colon\Aut(\Eman_{\pbs}) \times \Eman_{\pbs} \to \Eman_{\pbs}$, $\mu(A,\pv)=A\pv$.
Now, for each $\pv\in\Eman_{\pbs}$ define the following continuous function
\[
    \func_{\pv}\colon\Aut(\Eman_{\pbs})  \to \bR,
    \qquad
    \func_{\pv}(A)=\func(A(\pv)) - \func(\pv).
\]
Then $\Gy{\pbs} = \mathop{\cap}\limits_{\pv\in\Eman_{\pbs}} \func_{\pv}^{-1}(0)$, and so its is closed as an intersection of closed sets.
\end{proof}

Let $\GDistr = \{ \Gy{\pbs} \mid \pbs\in\Yman \}$ be a distribution of groups over a subset $\Yman\subset\fSing$.
Say that a bundle automorphism $\Alin\in\GLES$ \term{belongs to $\GDistr$} if $\restr{\Alin}{\Eman_{\pbs}} \in \Gy{\pbs}$ for every $\pbs\in\Yman$.
The set of all $\Alin\in\GLES$ belonging to $\GDistr$ will be denoted by $\GLESGG{\GDistr}$.
Evidently, $\GLESGG{\GDistr}$ is a subgroup of $\GLES$ and for every $\pbs\in\Yman$ we have the following \term{continuous} homomorphism
\[
    \tFibMap_{\pbs}\colon \GLESGG{\GDistr}\to\Gy{\pbs},
    \qquad
    \tFibMap_{\pbs}(\Alin) = \restr{\Alin}{\Eman_{\pbs}}.
\]

Let us mention that in general $\GLESGG{\GDistr}$ might have a very complicated local structure.
We will now consider an example of distributions with ``good'' local structure.

Let $\GDistr$ be a Lie groups distribution over a subset $\Yman\subset\fSing$ and $\Yman_0 = \{ \pbs\in\Yman \mid \dim\Gy{\pbs} = 0\}$ be the subset of $\Yman$ consisting points $\pbs$ for which the group $\Gy{\pbs}$ is zero-dimensional.
Then the following set
\begin{equation}\label{equ:zero_kernel}
   \zeroker{\GDistr} := \mathop{\cap}\limits_{\pbs\in\Yman_0}\tFibMap_{\pbs}^{-1}(\id_{\Eman_{\pbs}}) \equiv
                        \mathop{\cap}\limits_{\pbs\in\Yman_0} \ker(\tFibMap_{\pbs}),
\end{equation}
is a closed normal subgroup of $\GLESGG{\GDistr}$.
We will call $\zeroker{\GDistr}$ the \term{\dker} of $\GDistr$.

Notice that if $\dim\Gy{\pbs}=0$, then the one-point set $\{ \id_{\Eman_{\pbs}} \}$ is open in $\Gy{\pbs}$, and therefore $\tFibMap_{\pbs}^{-1}(\id_{\Eman_{\pbs}})$ is open in $\GLESGG{\GDistr}$.
However, if $\Yman_0$ is infinite, then in general $\zeroker{\GDistr}$ (being thus an intersection of infinitely many open sets) might be not open in $\GLESGG{\GDistr}$.

\begin{sublemma}\label{lm:G_contractible}
Let $\GDistr$ be a Lie groups distribution over a subset $\Yman\subset\fSing$.
\begin{enumerate}[leftmargin=*]
\item\label{enum:lm:0-uniform:subnbh}
The following conditions are equivalent:
\begin{enumerate}[label={\rm(1\alph*)}, leftmargin=*]
\item\label{enum:lm:0-uniform:ker0}
$\zeroker{\GDistr}$ is open (and thus open closed) in $\GLESGG{\GDistr}$;

\item\label{enum:lm:0-uniform:nbh_of_id}
$\zeroker{\GDistr}$ contains some (not necessarily open) neighborhood of $\id_{\Eman_{\pbs}}$ in $\GLESGG{\GDistr}$;

\item\label{enum:lm:0-uniform:nbh}
there exists an open neighborhood $\NbhGIdE$ of $\id_{\Eman}$ in $\GLESGG{\GDistr}$ such that for every $\pbs\in\Yman$ with $\dim\Gy{\pbs}=0$ and every $\Alin\in\NbhGIdE$ we have that $\restr{\Alin}{\Eman_{\pbs}} = \id_{\Eman_{\pbs}}$.
\end{enumerate}
In the case~\ref{enum:lm:0-uniform:nbh} one can choose $\NbhGIdE \subset \zeroker{\GDistr}$;

\item\label{enum:lm:0-uniform:suff_cond}
Each of the following conditions implies that $\zeroker{\GDistr}$ is open in $\GLESGG{\GDistr}$:
\begin{enumerate}[label={\rm(2\alph*)}, leftmargin=*]
\item\label{enum:lm:0-uniform:Y0_finite}
$\Yman_0$ is finite, i.e.\ $\dim\Gy{\pbs}=0$ only for finitely many $\pbs\in\Yman$;
\item\label{enum:lm:0-uniform:quotient_finite}
the quotient group $\GLESGG{\GDistr}/\zeroker{\GDistr}$ is finite;
\item\label{enum:lm:0-uniform:conn}
there exists a \term{connected} neighborhood $\NbhGIdE$ of $\id_{\Eman}$ in $\GLESGG{\GDistr}$;
\item\label{enum:lm:0-uniform:contr}
there exists a \term{contractible} neighborhood $\NbhGIdE$ of $\id_{\Eman}$ in $\GLESGG{\GDistr}$.
\end{enumerate}
Moreover, in the cases~\ref{enum:lm:0-uniform:conn} and~\ref{enum:lm:0-uniform:contr}, one can choose $\NbhGIdE\subset\zeroker{\GDistr}$.
\end{enumerate}
\end{sublemma}
\begin{proof}
Statement~\ref{enum:lm:0-uniform:subnbh} directly follows form the definitions.

\ref{enum:lm:0-uniform:Y0_finite}
If $\Yman_0$ is finite, then by~\eqref{equ:zero_kernel}, $\zeroker{\GDistr}$ is an intersection of finitely many open sets $\tFibMap_{\pbs}^{-1}(\id_{\Eman_{\pbs}})$.

\ref{enum:lm:0-uniform:quotient_finite}
Suppose $\GLESGG{\GDistr}/\zeroker{\GDistr}$ is finite, and let $\Kman_1=\zeroker{\GDistr},\Kman_2,\ldots,\Kman_n$ be corresponding  adjacent classes of $\zeroker{\GDistr}$.
Since $\zeroker{\GDistr}$ is closed, we have that each $\Kman_i$ is also closed, and therefore $\zeroker{\GDistr} = \GLESGG{\GDistr}\setminus\mathop{\cup}\limits_{i=2}^{n}\Kman_i$ is open.

\ref{enum:lm:0-uniform:conn}
If $\NbhGIdE$ is connected, then for each $\pbs\in\Yman$ with $\dim\Gy{\pbs}=0$ the image $\tFibMap_{\pbs}(\NbhGIdE)$ is contained in the path component $\{\id_{\Eman_{\pbs}}\}$ of $\id_{\Eman_{\pbs}}$ in $\Gy{\pbs}$.
In other words, $\restr{\Alin}{\Eman_{\pbs}}=\id_{\Eman_{\pbs}}$ for all $\Alin\in\NbhGIdE$.
Thus, $\NbhGIdE \subset \zeroker{\GDistr}$.

Finally, \ref{enum:lm:0-uniform:contr} implies~\ref{enum:lm:0-uniform:conn} and therefore openness of $\zeroker{\GDistr}$.
\end{proof}

One of the main observations of the present paper is that each of the properties~\ref{enum:lm:0-uniform:ker0}-\ref{enum:lm:0-uniform:nbh} is equivalent to~\ref{enum:lm:0-uniform:contr}, i.e.\ to local contractibility of $\zeroker{\GDistr}$, see Lemma~\ref{lm:char:0-uniform} below.

\section{Main result}
\subsection{Leaf preserving diffeomorphisms}
Let $\Foliation$ be a partition of $\Mman$.
It will be convenient to call its elements \term{leaves}.
Let also $\fSing$ be a proper submanfold of $\Mman$, $\Uman$ be an open neighborhood of $\fSing$, and $\dif\colon\Uman\to\Mman$ be a smooth embedding with $\dif(\fSing)=\fSing$, and $\Vman\subset\Uman$ be any subset.
Say that $\dif$ is \term{$\Foliation$-leaf preserving on $\Vman$} if $\dif(\omega\cap\Vman) \subset \omega$ for each leaf $\omega$ of $\Foliation$.
In particular, a diffeomorphism $\dif\in\Diff(\Mman)$ is \term{$\Foliation$-leaf preserving}, whenever $\dif(\leaf)=\leaf$ for all $\leaf\in\Foliation$.
Denote by $\Diff(\calF)$ the group of all $\Foliation$-leaf preserving diffeomorphisms and by
\[
    \mathcal{D}_{*}(\Foliation, -) := \mathcal{D}_{*}(\Mman, -) \cap \Diff(\calF)
\]
the intersections of the corresponding groups from~\eqref{equ:diagam_homot_eq_all_maps} with $\Diff(\calF)$, so we just replace $\Mman$ with $\Foliation$ in all the notations.
For instance, $\DiffFix{\Foliation}{\fSing} := \DiffFix{\Mman}{\fSing} \cap \Diff(\calF)$ and so on.

Assume further that $\fSing$ is $\Foliation$-saturated.
Then $\Diff(\calF) \subset \DiffInv{\Mman}{\fSing}$, whence
\[
\Diff(\Foliation) \equiv \DiffInv{\Foliation}{\fSing} := \DiffInv{\Mman}{\fSing} \cap \Diff(\Foliation).
\]

Note that for each $\pbs\in\fSing$ and each $\dif\in\Diff(\Foliation)$ such that $\dif(\pbs)=\pbs$ we have the isomorphism $\restr{(\tfib{\dif})}{\Eman_{\pbs}}\colon\Eman_{\pbs}\to\Eman_{\pbs}$ of the corresponding fiber $\Eman_{\pbs}$.
In particular, we get the following subgroup of $\Aut(\Eman_{\pbs})$:
\begin{equation}\label{equ:GDistroFol}
    \Gy{\pbs} :=
    \{
        \restr{(\tfib{\dif})}{\Eman_{\pbs}}
        \mid
        \dif\in\DiffFix{\Foliation}{\fSing}
    \},
\end{equation}
the distribution of groups over $\fSing$:
\[
    \GFolDistr := \{ \Gy{\pbs} \mid \pbs\in\fSing \},
\]
and therefore the subgroup $\GLESGG{\GFolDistr}$ of $\GLES$.
By definition, $\Alin\in\GLES$ belongs to $\GLESGG{\GFolDistr}$ if and only if for every $\pbs\in\fSing$ we have that $\restr{\Alin}{\Eman_{\pbs}}\in\Gy{\pbs}$, i.e.\ there exists $\dif\in\DiffFix{\Foliation}{\fSing}$ (depending on $\pbs$) such that $\restr{(\tfib{\dif})}{\Eman_{\pbs}} = \restr{\Alin}{\Eman_{\pbs}}\colon \Eman_{\pbs} \to \Eman_{\pbs}$.

Thus, by definition,
\begin{equation}\label{equ:tDFS_GLFol}
\tFibMap\bigl(\DiffFix{\Foliation}{\fSing}\bigr) \subset \GLESGG{\Foliation}.
\end{equation}

Also, for each closed subset $\Xman\subset\Mman\setminus\fSing$ and each $\pbs\in\fSing$ define the subgroups of $\Gy{\pbs}$:
\begin{equation}\label{equ:GDistroFol_X}
    \Gy{\pbs,\Xman} :=
    \{
        \restr{(\tfib{\dif})}{\Eman_{\pbs}}
        \mid
        \dif\in\DiffFix{\Foliation}{\Xman\cup\fSing}
    \},
\end{equation}
and the following distribution over $\fSing$:
\[
    \GDistr_{\Foliation,\Xman} := \{ \Gy{\pbs,\Xman} \mid \pbs\in\fSing \}.
\]

Our aim is to establish the following leaf-preserving counterpart of Theorem~\ref{th:isot_nbh}.
\begin{subtheorem}[cf.~\cite{KhokhliukMaksymenko:Nbh:2022}]
\label{th:KhM:Fol}
Let $\Mman$ be a manifold, $\Foliation$ a partition of $\Mman$, $\fSing \subset\Mman$ a compact proper $\Foliation$-saturated submanifold, and $\vbp\colon\Eman\to\fSing$ some regular neighborhood of $\fSing$ in $\Mman$.
\begin{equation}\label{equ:diagam_homot_eq_all_maps:fol}
    \begin{aligned}
        \xymatrix@C=0.1em@R=1em{
            & & & & & \fbox{$\GLESGG{\GFolDistr}$}            \\
            & \DiffFix[\vbp]{\Foliation}{\fSing}   \ar@{^(->}[dd] \ar@/^9pt/@{->}[rrrru]^-{\tFibMap} \ar@{^(-->}[rr]^-{\simeq} &&
            \DiffFix{\Foliation}{\fSing}           \ar@{^(->}[dd] \ar@{->}[rru]^-{\tFibMap}     \\
            \fbox{$\DiffNb{\Foliation}{\fSing}$}   \ar@/^5pt/@{^(->}[ur] \ar@/_5pt/@{^(->}[dr] \ar@{^(->}[rr]^(0.6){\simeq}  &&
            \DiffFix[1]{\Foliation}{\fSing}        \ar@{^(->}[ur] \ar@{^(->}[dr]    \\
            & \DiffInv[\vbp]{\Foliation}{\fSing}   \ar[d]^-{\tRestr}  \ar@{^(-->}[rr]^-{\simeq} &&
            \fbox{\fbox{$\Diff(\Foliation) =\DiffInv{\Foliation}{\fSing}$}}
                 \ar[d]^-{\tRestr}   \\
            & \fbox{$\Diff(\fSing)$} \ar@{=}[rr] && \Diff(\fSing) &&
        }
    \end{aligned}
\end{equation}
\begin{enumerate}[wide, itemsep=1ex, label={\rm(\arabic*)}]
\item\label{enum:th:KhM:Fol:HE}
If $\fSing$ has an \NGOOD{$\Foliation$}\ neighborhood in $\Eman$, then the homotopy $\Hhmt$ from Theorem~\ref{th:isot_nbh} leaves invariant $\Diff(\Foliation)$, so the inclusion of triples
\begin{equation}\label{equ:DMS_homot_equiv_fol}
\begin{array}{lcr}
    \bigl(
        \DiffInv[\vbp]{\Foliation}{\fSing}, &
        \DiffFix[\vbp]{\Foliation}{\fSing}, &
        \DiffNb{\Foliation}{\fSing}
    \bigr)
    \\ &  \cap \\
    \bigl(
        \Diff(\Foliation) &
        \DiffFix{\Foliation}{\fSing}, &
        \DiffFix[1]{\Foliation}{\fSing}
    \bigr)
\end{array}
\end{equation}
is a homotopy equivalence, see the arrows $\xmonoArrow{~\simeq~}$ in~\eqref{equ:diagam_homot_eq_all_maps:fol}.

\item\label{enum:th:KhM:Fol:Fibr}
Suppose in addition that $\GFolDistr$ is a Lie groups distribution and $\zeroker{\GFolDistr}$ is open in $\GFolDistr$.
Then in~\eqref{equ:diagam_homot_eq_all_maps:fol} two upper arrows $\tFibMap$ have local sections.

\item\label{enum:th:KhM:Fol:Restr}
Moreover, if the right vertical arrow $\tRestr\colon\Diff(\Foliation)\to\Diff(\fSing)$ has a local section, then the left arrow $\tRestr$ also has a local section.
In other words, every consecutive pair of vertical or diagonal arrows $\calP\monoArrow \calQ \to \calR$ constitute a fibration over its image.
\end{enumerate}

Also, if $\Xman\subset\Mman\setminus\fSing$ is a closed subset, then each of the statements~\ref{enum:th:KhM:Fol:HE}-\ref{enum:th:KhM:Fol:Restr} holds if we replace each group $\Diff_{*}(\Foliation, -)$ with the intersection $\Diff_{*}(\Foliation, -)\cap\DiffFix{\Mman}{\Xman}$ and also replace the symbol $\GFolDistr$ with $\GDistr_{\Foliation,\Xman}$.
\end{subtheorem}

Theorem~\ref{th:KhM:Fol} allows to relate homotopy groups of $\Diff(\Foliation)$ via long exact sequences of homotopy groups of the corresponding fibrations with the homotopy groups of the groups $\Diff(\fSing)$, $\GLESGG{\GFolDistr}$, and $\DiffNb{\Foliation}{\fSing}$.
Examples of concrete computations will be given in Section~\ref{sect:applications}.

\begin{subremark}\rm
Note that the two homotopy equivalences in~\eqref{equ:diagam_homot_eq_all_maps:fol} denoted by dashed arrows $\xymatrix{\ar@{^(-->}[r]^{\simeq} &}$ are proved in~\cite{KhokhliukMaksymenko:Nbh:2022}, and we included them into Theorem~\ref{th:KhM:Fol} just to complete all the picture.
The principal new statement in Theorem~\ref{th:KhM:Fol} is the construction of local sections for the upper arrow $\tFibMap$, which will be done in Section~\ref{sect:proof:th:KhM:Fol}.

Also note that the diagram~\eqref{equ:diagam_homot_eq_all_maps:fol} only partially mimics the non-foliated variant~\eqref{equ:diagam_homot_eq_all_maps}: there is a lack of an analogue of $\GLESGG{\GFolDistr}$ for vector bundle automorphisms non-fixed on $\fSing$.
We will not discuss that in the present paper.
\end{subremark}

\subsection{Morse-Bott foliations}
We will show here that Theorem~\ref{th:KhM:Fol} holds for a certain class of Morse-Bott foliations.
Recall that we allow manifolds (as well as submanifolds) to be non-connected and that their connected components may have distinct dimensions.

Let $\vbp\colon\Eman\to\fSing$ be a smooth vector bundle over a smooth manifold $\fSing$.
A continuous function $\gfunc\colon\Eman\to\bR$ is called \term{fiberwise homogeneous of degree $k\geqslant0$}, (or \term{$k$-homogeneous}) if for each $\pv\in\Eman$ and $t\geqslant0$ we have that $\gfunc(t\pv) = t^k\gfunc(\pv)$.
Say that a $2$-homogeneous function $\gfunc\colon\Eman\to\bR$ is \term{non-degenerate} if its restriction to each fiber $\Eman_{\pbs}=\vbp^{-1}(\pbs)$, $\pbs\in\fSing$, is a non-degenerate quadratic form.

\begin{subexample}\label{exmp:mb-func-extr}\rm
Suppose we are given a smooth fiberwise scalar product on $\Eman$ regarded as a $\Cinfty$ map $(\cdot,\cdot): \Eman\oplus\Eman\to\bR$.
Then the following Morse-Bott function $\func:\Eman\to\bR$, $\func(\pv)=(\pv,\pv)$, satisfies the above conditions: it is $2$-homogeneous and $\fSing$ is the critical submanifold of $\func$ corresponding to the minimum value $0$.
\end{subexample}

\begin{subdefinition}\label{def:mf-fol-homog}
Let $\fSing$ be a proper submanifold of a manifold $\Mman$.
A partition $\Foliation$ of $\Mman$ will be called a \term{Morse-Bott foliation linearizable near the singular manifold $\fSing$} if
\begin{enumerate}[leftmargin=*]
\item connected components of $\fSing$ are leaves of $\Foliation$;
\item the restriction $\restr{\Foliation}{\Mman\setminus\fSing}$ is a usual $\Cinfty$ foliation of codimension $1$, and
\item there exist
\begin{enumerate}[label={\rm(\alph*)}]
\item a regular neighborhood $\vbp\colon\Eman\to\fSing$;
\item\label{enum:def:mf-fol-homog:homog} a non-degenerate $2$-homogeneous function $\gfunc\colon\Eman\to\bR$;
\item\label{enum:def:mf-fol-homog:triv-holom} and an open neighborhood $\Uman \subset \Eman$ of $\fSing$ such that for each leaf $\omega\in\Foliation$ the function $\gfunc$ takes a constant value on $\omega\cap\Uman$ whenever that intersection is non-empty.
\end{enumerate}
\end{enumerate}
\end{subdefinition}

There is an extensive literature on the topology Morse and Morse-Bott functions and foliations as well as their applications, see e.g.~\cite{Frankel:DCT:1965,Sharko:MFAT:1995, Prishlyak:UMJ:2000, Prishlyak:UMJ:2002, KadzisaMimura:JFPTA:2011, MartinezAlfaroMezaSarmientoOliveira:CM:2016, MartinezAlfaroMezaSarmientoOliveira:JDE:2016,  HladyshPrishlyak:UMJ:2016, MartinezAlfaroMezaSarmientoOliveira:TMNA:2018,  BatistaCostaMezaSarmiento:JS:2018, HladyshPrishlyak:JMPAG:2019, MaciasVirgos_PereiraSaez:IM:2020, HutchingsNelson:AGT:2020, Gelbukh:CMJ:2021, Gelbukh:SSMH:2022} and references therein.

\begin{subremark}\rm
Note that the usual definition of Morse-Bott foliations, e.g.~\cite{ScarduaSeade:JDG:2009, Evangelista-Alvarado_etall:JS:2019} requires only that every point $\pbs\in\fSing$ is a non-degenerate critical point for the restriction of $\gfunc$ to the fiber $\Eman_{\pbs}$.
Let $n=\dim\Eman_{\pbs}$.
Then, by Morse-Bott lemma, e.g.~\cite{BanyagaHurtubise:EM:2004}, there exists an open neighborhood $\Vman$ of $\pbs$ in $\fSing$, an open neighborhood $\Wman$ of $0$ in $\bR^{n}$, and an open embedding $\phi\colon\Vman\times\Wman\subset\Eman$ such that
\begin{enumerate}[label={\roman*)}]
\item $\phi(\px,0)=\px$ for all $\px\in\Vman\subset\fSing$;
\item $\gfunc\circ\phi(\px, \pvi{1},\ldots,\pvi{n}) = -\pvi{1}^2-\cdots-\pvi{k}^2 + \pvi{k+1}^2+\cdots+\pvi{n}^2$ for some $k\in\{0,\ldots,n-1\}$.
\end{enumerate}
Thus, we can locally change the regular neighborhood of a part $\Vman$ of $\fSing$ near $\pbs$ so that the restriction of $\gfunc$ to the fibers will be $2$-homogenenous.

The principal additional assumptions of Definition~\ref{def:mf-fol-homog} are
\begin{itemize}
\item
the ``linearizability'' condition~\ref{enum:def:mf-fol-homog:homog} requiring existence of a ``global'' regular neighborhood $\vbp:\Eman\to\fSing$ and a $2$-homogeneous function $\gfunc\colon\Eman\to\bR$ whose level sets define $\Foliation$ near $\fSing$,
\item
and condition~\ref{enum:def:mf-fol-homog:triv-holom} requiring that $\gfunc$ ``distinguishes'' (by its values) the leaves of $\Foliation$ in $\Uman$, which can be regarded as a triviality of holonomy of the singular leaves, i.e.\ connected components of $\fSing$.
\end{itemize}
I do not know whether for each Morse-Bott function $\gfunc$ on a manifold with a critical submanifold $\fSing$ one can choose a regular neighborhood $\vbp\colon\Eman\to\fSing$ such that $\restr{\gfunc}{\Eman}$ will be $2$-homogeneous near $\fSing$.
The difficulties appear even in the following situation.
Suppose $\Uman,\Vman \subset \fSing$ are two open subsets with $\Uman\cap\Vman\neq\varnothing$ and $\vbp\colon\Eman\to\fSing$ is a regular neighborhood of $\fSing$ such that $\gfunc$ is $2$-homogeneous near $\vbp^{-1}(\Uman)\cap\fSing$ and $\vbp^{-1}(\Vman)\cap\fSing$.
Then it is not even clear whether \term{there exists another regular neighborhood $\vbp'\colon\Eman'\to\fSing$ of $\fSing$ such that $\gfunc$ is $2$-homogeneous near $(\vbp')^{-1}(\Uman\cup\Vman) \cap \fSing$.}
It the latter property were true, then it could be extended at least to all compact%
\footnote{The author is grateful to the anonymous Referee for mentioning case of compact manifolds.}
manifolds $\fSing$.
\end{subremark}

\begin{subtheorem}\label{th:main_res_mb}
Let $\Foliation$ be a Morse-Bott foliation on a manifold $\Mman$ being linearizable near the singular manifold $\fSing$, and $\Xman\subset\Mman\setminus\fSing$ be a closed subset.
Then
\begin{enumerate}[label={\rm(\arabic*)}]
\item\label{enum:th:main_res_mb:Fgood} there exists a regular neighborhood $\vbp\colon\Eman\to\fSing$ of $\fSing$ for which $\fSing$ has an \NGOOD{$\Foliation$}\ neighborhood,
\item\label{enum:th:main_res_mb:0-uniform} the corresponding distribution of group $\GDistr_{\Foliation}$, see~\eqref{equ:GDistroFol}, is a Lie groups distribution and its $0$-kernel $\zeroker{\GDistr_{\Foliation}}$ is open;
\item\label{enum:th:main_res_mb:sect} the map $\tRestr\colon\DiffInv{\Foliation}{\fSing}\to\Diff(\fSing)$ has a local section.
\end{enumerate}
Hence, all statements of Theorem~\ref{th:KhM:Fol} hold for $\Foliation$, i.e.\ the inclusions~\eqref{equ:DMS_homot_equiv_fol} are homotopy equivalences and all arrows $\tFibMap$ and $\tRestr$ in~\eqref{equ:diagam_homot_eq_all_maps:fol} have local sections.
\end{subtheorem}

\subsection{Structure of the rest of the paper}
In Section~\ref{sect:rem:proof:th:isot_nbh} we discuss the proof of Theorem~\ref{th:isot_nbh}, and in Section~\ref{sect:proof:th:main_res_mb} prove Theorem~\ref{th:main_res_mb}.
In Section~\ref{sect:scalable_partitions} we establish several properties of scalable partitions.
In Section~\ref{sect:vb_automorphisms} for a vector bundle $\vbp\colon\Eman\to\fSing$ we describe the construction of a fiberwise exponential map for the group of vector bundle automorphisms $\GLES$, and construct a specific contraction ``along one-parameter subgroups'' of a certain neighborhood of $\id_{\Eman}$ in $\GLES$.
As a consequence we also prove Lemma~\ref{lm:char:0-uniform} claiming that openness of $0$-kernel for a Lie groups distribution is equivalent to local contractibility of the group $\GLESGG{\GDistr}$.

Theorem~\ref{th:KhM:Fol} will be proved in Section~\ref{sect:proof:th:KhM:Fol}.
Finally, in Section~\ref{sect:applications} we apply Theorem~\ref{th:KhM:Fol} to computations of homotopy types of groups $\Diff(\Foliation)$ for certain foliations with only extreme singularities.
In particular, we prove Theorem~\ref{th:Omega_On_O2__DFol} which give a short proof of some results from~\cite{KhokhliukMaksymenko:JHRS:2023}.

\section{Remarks to the proof of Theorem~\ref{th:isot_nbh}}\label{sect:rem:proof:th:isot_nbh}
\begin{enumerate}[wide, itemsep=1ex, labelwidth=!, labelindent=0pt, label={\rm(\Alph*)}]
\item
For the middle vertical sequence $\DiffFix{\Mman}{\fSing} \monoArrow \DiffInv{\Mman}{\fSing} \to \Diff(\fSing)$ statement~\ref{th:isot_nbh:DfMSp_split} is exactly Theorem~\ref{th:DinvMS_split}, and it is included into the formulation of Theorem~\ref{th:isot_nbh} just for the completeness of all the picture.

\item\label{enum:rem:th:isot_nbh:proof_1}
Statement~\ref{th:isot_nbh:linearization} is a \term{linearization theorem} being mutually a multi-parameter, compactly supported, and actually even foliated variant of the well known statement that each diffeomorphism $\dif\colon\bR^n \to \bR^n$ with $\dif(0)=0$ is isotopic to the linear isomorphism induced by its Jacobi matrix at $0\in\bR^n$, as well as a variant of a more general theorem on existence of isotopy between regular neighborhoods of $\fSing$.
First let us recall the following easy extension of the Hadamard lemma to maps between total spaces of vector bundles:
\begin{lemma}[{\rm e.g.~\cite{KhokhliukMaksymenko:Nbh:2022}}]
\label{lm:Hadamard_for_vb}
Let $\vbp\colon\Eman\to\fSing$ be a vector bundle over a compact manifold $\fSing$, and $\dif\colon\Vman_{\dif}\to\Eman$ be a $\Cinfty$ map defined on some neighborhood $\Vman_{\dif}$ of $\fSing$ such that $\dif(\fSing)=\fSing$.
Then there exists a unique $\Cinfty$ map $\divhom_{\dif}\colon\Vman_{\dif}\times[0;1]\to\Eman$ such that
\begin{enumerate}[leftmargin=8ex, label={\rm(\alph*)}]
\item\label{enum:lm:Hadamard_for_vb:ht} $\dif(\tau\px) = \tau\divhom_{\dif}(\px,\tau)$, for all $(\px,\tau)\in\Vman_{\dif}\times[0;1]$;
\item\label{enum:lm:Hadamard_for_vb:h0} $\divhom_{\dif}(\px,0)=\tfib{\dif}(\px)$ for all $\px\in\Vman_{\dif}$;
\item\label{enum:lm:Hadamard_for_vb:sing_inv} $\divhom_{\dif}(\px,\tau)=\dif(\px)$ for all $(\px,\tau)\in\fSing\times[0;1]$.
\qedhere
\end{enumerate}
\end{lemma}
Due to~\ref{enum:lm:Hadamard_for_vb:ht} we have that $\divhom_{\dif}(\px,\tau\px)=\frac{1}{\tau}\dif(\tau\px)$ for $\tau>0$.
The main observation of Theorem~\ref{th:isot_nbh} is that in Lemma~\ref{lm:Hadamard_for_vb} one can do two additional technical things:
\begin{itemize}[leftmargin=*, label={$-$}]
\item modify $\divhom_{\dif}$ near $\partial\Vman_{\dif}$ so that it will extend by the identity to a $\Cinfty$ diffeomorphism of $\Mman$;
\item for each $\dif\in\DiffInv{\Mman}{\fSing}$ ``uniformly'' choose a neighborhood $\Vman_{\dif}$ of $\fSing$ such that $\dif(\Vman_{\dif}) \subset\Eman$, and ``uniformly'' modify the corresponding homotopy $\divhom_{\dif}$ as in the previous item.
\end{itemize}
More precisely, choose some orthogonal structure on the fibers of $\Eman$ and let $\nrm{\cdot}\colon\Eman\to[0;+\infty)$ be the corresponding norm.
Also fix any $\Cinfty$ function $\mu\colon\bR\to[0;1]$ such that $\mu=0$ on $[0;\aConst]$ and $\mu=1$ on $[\bConst;+\infty)$.
Then it is shown in~\cite{KhokhliukMaksymenko:Nbh:2022} that there exists a continuous function $\delta\colon\DiffInv{\Mman}{\fSing}\to(0;+\infty)$ such that if we define:
\begin{align*}
    &\phi\colon\DiffInv{\Mman}{\fSing}\times[0;1]\times\Eman\to[0;1],&
    &\phi(\dif,t,\px) = t + (1-t) \mu\bigl(\tfrac{\nrm{\px}}{\delta(\dif)}\bigr),
\end{align*}
then the following map $\Hhmt\colon\DiffInv{\Mman}{\fSing} \times [0;1]\to \DiffInv{\Mman}{\fSing}$
\[
    \Hhmt(\dif,t)(\px) =
    \begin{cases}
        \dif(\px), & \px\in\Mman\setminus\Eman, \\
        \divhom_{\dif}(\px, \phi(\dif,t,\px)\px), & \px\in\Eman,
    \end{cases}
\]
is a homotopy having the following properties:
\begin{enumerate}[label={(\alph*)}]
\item\label{enum:H:prop:1} $\Hhmt_1=\id_{\DiffInv{\Mman}{\fSing}}$;
\item\label{enum:H:prop:3} $\Hhmt_t(\DiffInv[\vbp]{\Mman}{\fSing}) \subset \DiffInv[\vbp]{\Mman}{\fSing}$;
\item\label{enum:H:prop:2} $\Hhmt_0(\DiffInv{\Mman}{\fSing}) \subset \DiffInv[\vbp]{\Mman}{\fSing}$;
\item\label{enum:H:prop:4} $\Hhmt_t(\DiffFix{\Mman}{\fSing}) \subset \DiffFix{\Mman}{\fSing}$;
\item\label{enum:H:prop:5} $\tfib{\Hhmt_t(\dif)} = \tfib{\dif}$.
\end{enumerate}
The first three properties mean that $\Hhmt$ is a \term{deformation of $\DiffInv{\Mman}{\fSing}$ into $\DiffInv[\vbp]{\Mman}{\fSing}$}, so the inclusion $\DiffInv[\vbp]{\Mman}{\fSing}\subset\DiffInv{\Mman}{\fSing}$ is a homotopy equivalence.

Moreover, condition~\ref{enum:H:prop:4} says that $\DiffFix{\Mman}{\fSing}$ is invariant under $\Hhmt$, and therefore the inclusion
\[
    \DiffFix[\vbp]{\Mman}{\fSing} \equiv
    \DiffInv[\vbp]{\Mman}{\fSing} \cap \DiffFix{\Mman}{\fSing}
    \ \subset \
    \DiffFix{\Mman}{\fSing}
\]
is also a homotopy equivalence.

Finally, condition~\ref{enum:H:prop:5} implies, in particular, that $\Hhmt_t(\DiffFix{\Mman}{\fSing, 1}) \subset \DiffFix{\Mman}{\fSing, 1}$ for all $t\in[0;1]$.
Indeed, if $\dif\in\DiffFix{\Mman}{\fSing, 1}$, i.e.\ $\tfib{\dif}=\id_{\Eman}$, then by~\ref{enum:H:prop:5}, $\tfib{\Hhmt_t(\dif)} = \tfib{\dif}=\id_{\Eman}$ for all $t\in[0;1]$, and thus $\Hhmt_t(\dif)\in\DiffFix{\Mman}{\fSing, 1}$ as well.
Hence, the inclusion
\[
    \DiffNb{\Mman}{\fSing}
    \equiv
    \DiffFix[\vbp]{\Mman}{\fSing} \cap \DiffFix{\Mman}{\fSing, 1}
    \ \subset \
    \DiffFix{\Mman}{\fSing, 1}
\]
is a homotopy equivalence (which was not mentioned in~\cite{KhokhliukMaksymenko:Nbh:2022}).

\item
Statements concerning dashed arrows are not mentioned in~\cite{KhokhliukMaksymenko:Nbh:2022}, however their proofs are implicitly presented therein.

Namely, let us prove~\ref{th:isot_nbh:DfMSp_split} for dashed arrows $\tFibMap$.
Notice that $\DiffInv[\vbp]{\Mman}{\fSing} \subset \DiffInv{\Mman}{\fSing}$ and both these groups are mapped into $\GLE$ with the same map $\tFibMap$.
This implies that the local section $\GLE \supset  \NbhId  \xrightarrow{~s~} \DiffInv[\vbp]{\Mman}{\fSing}$ of the mapping $\tFibMap\colon \DiffInv[\vbp]{\Mman}{\fSing} \to \GLE$ \term{indeed constructed in~\cite{KhokhliukMaksymenko:Nbh:2022}}, is automatically a section of $\tFibMap\colon \DiffInv{\Mman}{\fSing} \to \GLE$.
Moreover, the restriction of $s$ to $\NbhId\cap\GLES$ is a local section of $\tFibMap\colon\DiffFix{\Mman}{\fSing} \to \GLES$.
\end{enumerate}

\section{Proof of Theorem~\ref{th:main_res_mb}}\label{sect:proof:th:main_res_mb}
Statement~\ref{enum:th:main_res_mb:Fgood} follows from Lemma~\ref{lm:example_Fgood}, and~\ref{enum:th:main_res_mb:sect} is proved in~\cite[Theorem~4.3]{KhokhliukMaksymenko:IndM:2020}.

\ref{enum:th:main_res_mb:0-uniform}
Let $\vbp\colon\Eman\to\fSing$, $\gfunc\colon\Eman\to\bR$, and $\Uman$ be the same as in Definition~\ref{def:mf-fol-homog}.
Due to Lemma~\ref{lm:G_contractible}\ref{enum:lm:0-uniform:quotient_finite} it suffices to show that {\em for each $\pbs\in\fSing$ the group
\[
    \Gy{\pbs} :=
    \{
        \restr{(\tfib{\dif})}{\Eman_{\pbs}}
        \mid
        \dif\in\DiffFix{\Foliation}{\fSing}
    \}
\]
is closed in $\Aut(\Eman_{\pbs})$ and $\zeroker{\GFolDistr}$ has finite index in $\GLESGG{\GFolDistr}$.}
We need few lemmas.

\begin{lemma}\label{lm:mbfol_prop}
Let $\dif\in\Diff(\Foliation)$.
Then the following statements hold.
\begin{enumerate}[leftmargin=*, label={\rm(\alph*)}]
\item\label{enum:lm:mbfol_prop:h_pres_F}
Let $\Uman':=\dif^{-1}(\Uman)\cap\Uman$ be the maximal open neighborhood of $\fSing$ in $\Uman$ such that $\dif(\Uman')\subset\Uman$.
Then $\gfunc(\dif(\px))=\gfunc(\px)$ for all $\px\in\Uman'$.

\item\label{enum:lm:mbfol_prop:Th_pres_F}
There exists a smaller neighborhood $\Uman'' \subset \Uman'$ of $\fSing$ such that
\begin{equation}\label{equ:g_Th__g}
    \gdif(\tfrac{1}{\tau}\dif(\tau\px))= \gdif(\tfib{\dif}(\px))= \gdif(\px)
\end{equation}
for all $\tau\in[0;1]$ and $\px\in\Uman''$.
\end{enumerate}
\end{lemma}
\begin{proof}
\ref{enum:lm:mbfol_prop:h_pres_F}
Let $\px\in\Uman'$ and $\omega$ be the leaf of $\Foliation$ containing $\px$.
Then $\px \in \omega\cap\Uman' \subset \omega\cap\Uman$ and $\dif(\px) \in \omega \cap \dif(\Uman') \subset \omega\cap\Uman$ as well.
Since $\gfunc$ takes a constant value on all of $\omega\cap\Uman$, we obtain that $\gfunc(\dif(\px))=\gfunc(\px)$.

\ref{enum:lm:mbfol_prop:Th_pres_F}
By Lemma~\ref{lm:Hadamard_for_vb} there exists a $\Cinfty$ map $\divhom\colon\Uman'\times[0;1]\to\Eman$ such that $
\divhom(\px,\tau\px) = \tfrac{1}{\tau}\dif(\tau\px)$ and $\divhom(\px,0)=\tfib{\dif}(\px)$ for all $\px\in\Uman'$ and $\tau\in(0;1]$.
Then the identity~\eqref{equ:g_Th__g} means that $\gdif\bigl(\divhom(\px,\tau\px)\bigr)=\gdif(\px)$.

Since $\divhom(\fSing\times[0;1])=\fSing$ and $\fSing$ is compact, there exists another neighborhood $\Uman'' \subset \Uman'$ with compact closure $\overline{\Uman''}$ such that $\divhom(\overline{\Uman''}\times[0;1]) \subset\Uman$.
Then for all $\px\in\Uman''$ and $\tau\in(0;1]$ we have that
\[
    \gfunc\bigl(\tfrac{1}{t}\dif(t\px)\bigr) =
    \tfrac{1}{t^2}\gfunc\bigl(\dif(t\px)\bigr) \stackrel{\ref{enum:lm:mbfol_prop:h_pres_F}}{=}
    \tfrac{t^2}{t^2}\gfunc(\px) =
    \gfunc(\px),
\]
i.e.\ $\gdif\bigl(\divhom(\px,\tau\px)\bigr)=\gdif(\px)$.
Hence, $\gdif(\px) = \gdif\bigl(\divhom(\px,0)\bigr) = \gfunc\bigl( \tfib{\dif}(\px)\bigr)$ by continuity of $\gdif$.
\end{proof}

\begin{lemma}\label{lm:local_triv_2form}
Let $\pbs\in\fSing$, $n=\dim\Eman_{\pbs}$, and $k$ be the index of the non-degenerate $2$-form $\restr{\gfunc}{\Eman_{\pbs}}:\Eman_{\pbs} \to \bR$.
Then there exist an open neighborhood $\Vman$ of $\pbs$ and a local trivialization $\phi:\Vman\times\bR^{n}\to\Eman$ of $\vbp$ over $\Vman$ such that $\psi(\px,0)=(\px,0)$ and
\begin{equation}\label{equ:quadr_form_ind_k}
    \gfunc\circ\phi(\px,\pvi{1},\ldots,\pvi{n}) = -\pvi{1}^2-\cdots-\pvi{k}^{2}+\pvi{k+1}^2+\cdots+\pvi{n}^{2}
\end{equation}
for all $\px\in\Vman$ and $\pv=(\pvi{1},\ldots,\pvi{n})\in\bR^{n}$.
\end{lemma}
\begin{proof}
Let $\phi'\colon\Vman'\times\bR^{n} \to \Eman$ be any local trivialization of the vector bundle $\vbp$ over some open \term{connected} neighborhood $\Vman'$ of $\pbs$ in $\fSing$.
Thus $\phi$ is an open $\Cinfty$ embedding such that for every $\px\in\Vman'$ we have that $\phi(\px\times\bR^n)=\Eman_{\px}$ and the induced map $\phi_{\px}:\bR^n\to\Eman_{\px}$, $\pv\mapsto\phi(\px,\pv)$, is a linear isomorphism.

Since $\gfunc$ is $\Cinfty$ and fiberwise homogeneous, it is well known, see e.g.\ discussion in~\cite{Dubovski:MOF:2017} and also~\cite{Maksymenko:Lens2:2022}, that $\gfunc\circ\phi:\Vman'\times\bR^{n} \to \bR$ is given by $\gfunc\circ\phi'(\px,\pvi{1},\ldots,\pvi{n}) = \sum\limits_{i,j=1}^{n} a_{ij}(\px) \pvi{i}\pvi{j}$ for some $\Cinfty$ functions $a_{ij}:\Vman'\to\bR$ such that $a_{ij}=a_{ji}$, i.e.\ it is fiberwise polynomial.
Since as each map $\phi_{\px}$, $\px\in\Vman'$, is an isomorphism and $\gfunc|_{\Eman_{\px}}$ is non-degenerate quadratic form, the corresponding quadratic forms $\restr{\gfunc\circ\phi}{\px\times\bR^n}$ are non-degenerate as well.
Moreover, since $\Vman'$ is connected, all those quadratic forms have the same index, say $k$.

It further follows from the proof of Morse-Bott lemma, e.g.~\cite{BanyagaHurtubise:EM:2004} (being just a fiberwise variant of the reduction of a quadratic form to the sum of squares), that one can find a smaller neighborhood $\Vman \subset \Vman'$ of $\pbs$ in $\fSing$ and a trivial \term{vector bundle isomorphism} $\psi\colon\Vman\times\bR^n \to \Vman\times\bR^{n}$, such that $\psi(\px,0)=(\px,0)$ and $\gfunc\circ\phi'\circ\psi(\px,\pvi{1},\ldots,\pvi{n}) = - \smashoperator{\sum\limits_{i=1}^{k}}\pvi{i}^2 + \smashoperator{\sum\limits_{i=k+1}^{n}}\pvi{i}^2$ for all $\px\in\Vman$.
Note that the principal point here is that $\psi\colon\{\px\}\times\bR^{n}\to\{\px\}\times\bR^{n}$ is a \term{linear isomorphism} for each $\px\in\Vman$, i.e.\ we can reduce $\gfunc\circ\phi'$ to the sum of squares mutually at all points $\px\in\Vman$ by fiberwise linear change of coordinates.
It remains to put $\phi = \phi'\circ\psi$.
\end{proof}

\begin{lemma}\label{lm:mbfol:Gy_is_closed}
For each $\pbs\in\Vman$ the group $\Gy{\pbs}$ is closed in $\Aut(\Eman_{\pbs})$ and has dimension $\tfrac{n(n-1)}{2}$, where $n=\dim\Eman_{\pbs}$.
Moreover, if $n=1$, i.e.\ $\dim\Gy{\pbs} = 0$, then $\Gy{\pbs}$ is a subgroup of $\{\pm\id_{\Eman_{\pbs}}\}$.
\end{lemma}
\begin{proof}
Let $\Alin\in\Gy{\pbs}$ be any element, so there exists $\dif\in\Diff(\Foliation)$ such that $\dif(\pbs)=\pbs$ and $\Alin=\restr{(\tfib{\dif})}{\Eman_{\pbs}}$.
Then by~\eqref{equ:g_Th__g}, $\Alin$ preserves $\restr{\gfunc}{\Eman_{\pbs}}$.

Hence, if $\phi\colon\Vman\times\bR^n\to\Eman$ is a trivialization of $\vbp$ over $\Vman$ as in Lemma~\ref{lm:local_triv_2form}, then in this chart $\Gy{\pbs}$ can be identified with the subgroup of $\GLR{n}$ preserving the quadratic form~\eqref{equ:quadr_form_ind_k}, i.e.\ with the indefinite orthogonal group $\Ort(k,n-k)$.

It is well known that $\dim\Ort(k,n-k) = \tfrac{n(n-1)}{2}$, so it does not depend on $k$, and $\Ort(k,n-k)$ has finitely many path components (actually either $2$ or $4$), and its identity path component is denoted by $\SO^{+}(n,k-n)$.

Consider two cases.

1) Let $n=1$, so $\gfunc\circ\phi(\px,\pv) = \pm\pv^2$, and $\Ort(0,1)=\Ort(1,0) = \Ort(1) = \{ \pm\id_{\bR}\}$.
Hence, $\Gy{\pbs}$ is either $\{\id_{\bR}\}$ or $\Ort(1)$, and therefore it is closed in $\GLR{1}$.

2) Suppose $n\geqslant2$, so $\dim\Ort(k,n-k)=\tfrac{n(n-1)}{2} \geqslant1$.
\term{We claim that $\Gy{\pbs}$ contains a neighborhood $\UU$ of the unit $(n\times n)$-matrix $I$ in $\Ort(k,n-k)$}.
Since any such neighborhood generates the identity path component of $\SO^{+}(n,k-n)$ this will imply that
\[ \SO^{+}(n,k-n) \ \subset \ \Gy{\pbs} \ \subset \ \Ort(k,n-k). \]
Whence, $\Gy{\pbs}$ will be a union of (finitely many) connected components of $\Ort(k,n-k)$, and therefore it will be closed in $\GLR{n}$.

\newcommand\frso{\frak{so}(k,n-k)}
Indeed, let $\frso \subset \MatR{n}$ be the tangent space to $\SO^{+}(k,n-k)$ at the unit matrix $I$, i.e.\ the Lie algebra of $\SO^{+}(k,n-k)$.
Then we have the usual exponential map
\[
    \exp\colon\frso\to\SO^{+}(k,n-k),
    \qquad
    \exp(\Alin) = e^{\Alin} = \sum\limits_{i=0}^{\infty} \frac{\Alin^i}{i!},
\]
which yields a diffeomorphism of some open neighborhood $\WW$ of $0\in\frso$ onto some open neighborhood $\VV$ of $I$ in $\SO^{+}(k,n-k)$.

For $\eps>0$ let $\Dr{\eps} \subset \bR^{n}$ be a closed disk of radius $\eps$.
Fix two $\Cinfty$ functions $\mu:\Vman \to[0;1]$ and $\lambda:\bR^{n}\to[0;1]$ such that
\begin{enumerate}[label={$\bullet$}]
\item
$\supp(\mu)$ is compact and $\mu(\pbs)=1$;
\item
$\lambda=1$ on $\Dr{\eps}$ and $\lambda=0$ on $\bR^{n}\setminus\Dr{2\eps}$ for some $\eps>0$ such that $\phi\bigl(\Vman\times\Dr{3\eps}\bigr) \subset \Uman$.
\end{enumerate}
Define the following $\Cinfty$ map $G\colon\WW \times\Vman\times\bR^{n} \to \Vman\times\bR^{n}$ by
\begin{equation}\label{equ:map_GB}
    G(\Blin,\px,\pv) =
    \bigl(
        \px,
        e^{\mu(\px)\lambda(\pv) \Blin}\cdot\pv
    \bigr).
\end{equation}
For each $\Blin\in\WW$ it will be convenient to consider the map $G_{\Blin}\colon \Vman\times\bR^{n} \to \Vman\times\bR^{n}$ by $G_{\Blin}(\px,\pv) = G(\Blin,\px,\pv)$.
Then the following statements hold.
\begin{enumerate}[label={\rm(G\arabic*)}, topsep=1ex, itemsep=1ex, leftmargin=*]
\item\label{enum:GB:support}
$\supp(G_{\Blin}) \subset \supp(\mu)\times \Dr{2\eps}$, i.e.\ $G_{\Blin}$ is fixed out of $\supp(\mu)\times \Dr{2\eps}$.

\item\label{enum:GB:G0}
$G_{0} = \id_{\Vman\times\bR^{n}}$.

\item\label{enum:GB:TGB}
Let $\px\in\Vman$.
Then for each $\pv\in\bR^{n}$ with $\nrm{\pv}<\eps$ we have that $\lambda(\pv)=1$ and therefore
\[
    G_{\Blin}(\px,\pv) = \bigl(\px, e^{\mu(\px)\Blin}\cdot\pv\bigr).
\]
Hence, $\restr{(\tfib{G_{\Blin}})}{\px\times\bR^{n}} = e^{\mu(\px)\Blin}:\bR^{n}\to\bR^{n}$ belongs to $\SO^{+}(k,n-k)$.
In particular, since $\mu(\pbs)=1$ we have that
\begin{equation}\label{equ:TGB_y=eB}
    \restr{(\tfib{G_{\Blin}})}{\pbs\times\bR^{n}} = e^{\Blin}:\bR^{n}\to\bR^{n}
\end{equation}

\item\label{enum:GB:gGB_g}
$\gfunc\circ\phi\circ G_{\Blin}(\px,\pv) = \gfunc\circ\phi(\px,\pv)$ for all $(\px,\pv)\in\Vman\times\bR^{n}$.
This follows from formula~\eqref{equ:map_GB} for $G_{\Blin}$ and the observation that if $\Blin\in\frso$, then $\tau\Blin\in\frso$ for all $\tau\in[0;1]$ as well, and therefore $e^{\tau\Blin} \in \SO^{+}(k,n-k)$.
\end{enumerate}

Since $G$ is continuous, it follows from~\ref{enum:GB:support} and~\ref{enum:GB:G0} that there exists an open neighborhood $\WW'$ of $0\in\frso$ such that $G_{\Blin}$ is a diffeomorphism of $\Vman\times\bR^{n}$.
Moreover, then by~\ref{enum:GB:gGB_g} the map
\[
    \dif_{\Blin}\colon\Mman\to\Mman,
    \qquad
    \dif(\px) =
    \begin{cases}
        \phi\circ G_{\Blin}\circ \phi^{-1}(\px), & \px\in\phi(\Vman\times\bR^{n}), \\
        \px, & \text{otherwise},
    \end{cases}
\]
is a diffeomorphism supported in some neighborhood of $\pbs$ in $\Eman$ and satisfying $\gfunc\circ\dif=\gfunc$ on $\Eman$.
In particular, $\dif_{\Blin}\in\DiffFix{\Foliation}{\fSing}$.

Now let $\VV' = \exp(\WW')$ be the corresponding neighborhood of $I$ in $\SO^{+}(k,n-k)$.
Then by~\eqref{equ:TGB_y=eB}, $\VV' = \{ \restr{(\tfib{\dif_{\Blin}})}{\Eman_{\pbs}} \mid \Blin\in\WW' \} \subset \Gy{\pbs}$.
\end{proof}

\begin{lemma}
The quotient group $\GLESGG{\GFolDistr}/\zeroker{\GFolDistr}$ is finite.
\end{lemma}
\begin{proof}
Let $\Yman_0 = \{ \pbs\in\fSing \mid \dim\Gy{\pbs} = 0\}$.
Recall that for every $\pbs\in\fSing$ we have the natural homomorphism $\tFibMap_{\pbs}:\GLESGG{\GDistr}\to\Gy{\pbs}$, $\tFibMap_{\pbs}(\Alin) = \Alin|_{\Eman_{\pbs}}$, and
\[
    \zeroker{\GDistr} :=
    \{ \Alin\in\GLESGG{\GDistr} \mid \restr{\Alin}{\Eman_{\pbs}} = \id_{\Eman_{\pbs}} \ \text{for all} \ \pbs\in\Yman_0 \}.
\]

By Lemma~\ref{lm:mbfol:Gy_is_closed} the dimension of the group $\Gy{\pbs}$ depends on the dimension of the fiber over $\pbs$.
Hence, if $\fSing'$ is a connected component of $\fSing$, then all fibers over points $\pbs\in\fSing'$ have the same dimension, and therefore $\dim\Gy{\pbs}$ is also the same for all $\pbs\in\fSing'$.
In particular, $\Yman_0$ is a union of connected components, say $\fSing_1,\ldots,\fSing_l$ of $\fSing$.

Moreover, let $\Eman_i$ be the connected component of $\Eman$ containing $\fSing_i$, so $\restr{\vbp}{\Eman_i}\colon\Eman_i\to\fSing_i$ is a regular neighborhood of $\fSing_i$.
Let also $\pbs\in\fSing_i$, so $\dim\Gy{\pbs} = 0$.
Then for each $\Alin\in\GLESGG{\GDistr}$ we have that $\restr{\Alin}{\Eman_{\pbs}} = \eps_{\pbs}\id_{\Eman_{\pbs}}$ for some $\eps_{\pbs}\in\{\pm1\}$.
It follows from continuity of $\Alin$ that $\eps_{\pbs}$ is constant on $\fSing_i$, and we will denote that common value of all $\eps_{\pbs}$ by $\eps_i(\Alin)$.
Thus, $\restr{\Alin}{\Eman_i} = \eps_i(\Alin)\id_{\Eman_i}$.

Let $\bZ_2 = \{\pm 1\}$ be the cyclic group of order $2$.
Then we get a natural homomorphism
\[
    \eta:\GLESGG{\GDistr} \to (\bZ^2)^{l},
    \qquad
    \eta(\Alin) =
    \bigl(
        \eps_1(\Alin), \ldots, \eps_l(\Alin),
    \bigr)
\]
and its kernel is evidently $\zeroker{\GFolDistr}$.
Hence, $\GLESGG{\GFolDistr}/\zeroker{\GFolDistr}$ is isomorphic with a subgroup of a finite group $(\bZ^2)^{l}$ and therefore is finite.
\end{proof}

\section{Several properties of scalable partitions}\label{sect:scalable_partitions}

\begin{lemma}\label{lm:Fgood_nbh_props}
Let $\Foliation$ be a partition of a manifold $\Mman$, $\fSing$ be a proper $\Foliation$-saturated submanifold of $\Mman$, and $\Uman\subset\Eman$ be an open neighborhood of $\fSing$.
\begin{enumerate}[wide]
\item\label{enum:lm:Fgood_nbh_props:subnbh}
If $\Uman$ is an $\Foliation$-isolating (resp.\ $\Foliation$-scalable), then so is any other open neighborhood $\Vman\subset\Uman$ of $\fSing$ contained in $\Uman$.

\item\label{enum:lm:Fgood_nbh_props:pres_on_V__pres_on_U}
Suppose $\Uman$ is an $\Foliation$-scalable neighborhood of $\fSing$.
Let also $\Alin \in \GLES$ be such that for every $\pbs\in\fSing$ there exists an open neighborhood $\Vman_{\pbs}$ in the fiber $\Eman_{\pbs}$ satisfying
\begin{equation}\label{equ:A_leaf_fiber_Vy}
\Alin(\omega \cap \Vman_{\pbs})  \subset \omega \cap \Uman
\end{equation}
for all leaves $\omega\in\Foliation$.
Then $\Alin$ preserves leaves of $\Foliation$ on $\Uman_{\Alin}:=\Alin^{-1}(\Uman)\cap\Uman$.

Note that we do not assume any kind of continuity of $\Vman_{\pbs}$ in $\pbs$, and that~\eqref{equ:A_leaf_fiber_Vy} holds if $\Alin$ preserves leaves of $\Foliation$ on some open neighborhood of $\fSing$ in $\Eman$.

\item\label{enum:lm:Fgood_nbh_props:tfib}
Suppose $\Uman$ is star-convex and $\Foliation$-scalable for $\fSing$.
Let also $\dif\in\DiffFix{\Foliation}{\fSing}$, $\Vman\subset\Uman$ be a neighborhood of $\fSing$ such that $\dif(\Vman)\subset\Eman$, and $\divhom\colon\Vman\times[0;1]\to\Eman$ be the $\Cinfty$ homotopy from Lemma~\ref{lm:Hadamard_for_vb}.
Then there exists another neighborhood $\Wman\subset\Vman$ of $\fSing$ such that for every $\tau\in(0;1]$ the map
\[
    \dif_{\tau}\colon\Wman\to\Eman,
    \qquad
    \dif_{\tau}(\px):=\divhom(\px,\tau) = \tfrac{1}{\tau}\dif(\tau\px),
\]
preserves leaves of $\Foliation$ in $\Wman$.

Moreover, if $\Uman$ is also $\Foliation$-isolating for $\fSing$, then $\Alin:=\tfib{\dif}=\divhom(\cdot,0)$ also preserves leaves of $\Foliation$ in $\Wman$, and therefore due to~\ref{enum:lm:Fgood_nbh_props:pres_on_V__pres_on_U}, in $\Uman_{\Alin}:=\Alin^{-1}(\Uman)\cap\Uman$.

\item\label{enum:lm:Fgood_nbh_props:clGLESGG}
Suppose $\Uman$ is \NGOOD{$\Foliation$}.
Let also $\clGBunchFol := \{ \overline{\Gy{\pbs}} \mid \pbs\in\fSing \}$ be the closed distribution consisting of closures of groups from $\GFolDistr$, see~\eqref{equ:GDistroFol}.
Then each $\Alin\in\GLESGG{\clGBunchFol}$ preserves leaves of $\Foliation$ in $\Uman_{\Alin}:=\Alin^{-1}(\Uman)\cap\Uman$, (though it is not claimed that there exists $\dif\in\DiffFix{\Foliation}{\fSing}$ such that $\Alin=\tfib{\dif}$).
\end{enumerate}
\end{lemma}
\begin{proof}
Statement~\ref{enum:lm:Fgood_nbh_props:subnbh} directly follows from the definitions.

\ref{enum:lm:Fgood_nbh_props:pres_on_V__pres_on_U}
First let us note that $\Uman_{\Alin}$ is the maximal open subset of $\Uman$ such that $\Alin(\Uman_{\Alin})\subset\Uman$.

Now let $\px\in\Uman_{\Alin}$ be any point, so $\Alin(\px) \in \Uman$ as well.
We should show that $\px$ and $\Alin(\px)$ belong to the same leaf of $\Foliation$.

Indeed, since $\Vman_{\pbs}$ is a neighborhood of $\pbs$ in $\Eman_{\pbs}$, and $\Alin(\pbs)=\pbs$, there exists $\tau>0$ such that $\tau\px\in\Vman_{\pbs}$ and $\Alin(\tau\px) \in \Uman$.
Let $\omega$ be the leaf of $\Foliation$ containing $\tau\px$, so $\tau\px\in\omega\cap\Vman_{\pbs}$.
Then by~\eqref{equ:A_leaf_fiber_Vy}, $\tau \Alin(\px) = \Alin(\tau\px)\in\omega \cap \Uman$.
Thus, we get the following four points belonging to $\Uman$:
\[
    \tau\px,                             \qquad
    \tau \Alin(\px) = \Alin(\tau\px),            \qquad
    \px = \tfrac{1}{\tau} \tau\px,       \qquad
    \Alin(\px) = \tfrac{1}{\tau} \Alin(\tau\px).
\]
By the construction, the first two ones belong to the same leaf $\omega\in\Foliation$.
As $\Uman$ is $\Foliation$-scalable, $\px$ and $\Alin(\px)$ must also belong to the same leaf of $\Foliation$.

\ref{enum:lm:Fgood_nbh_props:tfib}
Since $\divhom(\fSing\times[0;1])=\fSing$, there exists an open star-convex neighborhood $\Wman\subset\Vman$ of $\fSing$ with compact closure such that $\divhom(\Wman\times[0;1])\subset\Uman$.
We claim that \term{$\dif_{\tau}$ preserves leaves of $\Foliation$ on $\Wman$ for all $\tau\in[0;1]$}, i.e.\ if $\px\in\Wman$ and $\omega$ is a leaf of $\Foliation$ containing $\px$, then $\dif_{\tau}(\px)\in\omega$ as well.

Indeed, first assume that $\tau\in(0;1]$.
Since $\Uman$ is star-convex, we have that $\tau\px\in\Uman$.
Further, as $\dif$ preserves the leaves of $\Foliation$, we also have that $\tau\px$ and $\dif(\tau\px)$ belong to the same leaf, say $\omega'$, of $\Foliation$.
Moreover, $\dif(\tau\px) = \tau\divhom(\px,\tau)\in \tau\Uman \subset\Uman$.
Finally, $\divhom(\px,\tau)\in\divhom(\Wman\times[0;1]) \subset \Uman$.
Thus, we get that the following four points belong to $\Uman$:
\[
    \tau\px,                       \qquad
    \dif(\tau\px),                 \qquad
    \px = \tfrac{1}{\tau} \tau\px, \qquad
    \dif_{\tau}(\px)=\tfrac{1}{\tau}\dif(\tau\px) = \divhom(\px,\tau).
\]
By the construction, the first two ones belong to the same leaf $\omega'\in\Foliation$.
As $\Uman$ is $\Foliation$-scalable, $\px$ and $\dif_{\tau}(\px)$ must also belong to the same leaf of $\Foliation$.
Hence, $\dif_{\tau}(\px)\in\omega$.

Consider the case $\tau=0$ and assume that $\Uman$ is $\Foliation$-isolating for $\fSing$.
By the construction $\divhom\bigl( \{\px\}\times (0;1] \bigr) \subset \omega$.
Suppose $\divhom(\px,0)\not\in\omega$.
Then $\divhom(\px,0) = \tfib{\dif}(\px) \in \Uman\cap (\overline{\omega}\setminus\omega) \subset \fSing$, since $\Uman$ is $\Foliation$-isolating for $\fSing$.
But the restriction of $\divhom(\cdot,0) = \tfib{\dif}$ to $\Eman_{\px}$ is a linear map, and $\px$ is a zero in $\Eman_{\px}$.
Hence, $\divhom(\px,0)=\px \in \omega$ as well, which contradict to the assumption.

\ref{enum:lm:Fgood_nbh_props:clGLESGG}
Suppose $\Uman$ is \NGOOD{$\Foliation$}\ and let $\Alin\in\GLESGG{\GFolDistr}$.
We should prove that $\Alin$ preserves orbits of $\Foliation$ on $\Uman_{\Alin}$.
By~\ref{enum:lm:Fgood_nbh_props:pres_on_V__pres_on_U} it suffices to show that for every $\pbs\in\fSing$ there exists an open neighborhood $\Vman_{\pbs}$ of $\pbs$ in the fiber $\Eman_{\pbs}$ such that $\Alin(\omega \cap \Vman_{\pbs})  \subset \omega \cap \Uman$.
The arguments are similar to~\ref{enum:lm:Fgood_nbh_props:tfib}.

Note that the assumption $\Alin\in\GLESGG{\GFolDistr}$ means that $\Alin|_{\Eman_{\pbs}} \in \overline{\Gy{\pbs}}$, so there exists a sequence $\{ \dif_i\}_{i\in\bN} \subset \DiffFix{\Foliation}{\fSing}$ such that the corresponding sequence $\{\Alin_i:=\tfib{\dif_i}\}_{i\in\bN} \subset \Gy{\pbs}$ converges to $\Alin$.
Then by~\ref{enum:lm:Fgood_nbh_props:tfib} each $\Alin_i$ preserves leaves of $\Foliation$ in $\Uman_{\Alin_i}:=\Alin_i^{-1}(\Uman) \cap \Uman$.

Since each $\Alin_i$ are linear automorphisms of a finite-dimensional vector space $\Eman_{\pbs}$, say of dimension $n$, that convergence is equivalent to the pointwise convergence of the values $\Alin_i(\pui{j}) \to \Alin(\pui{j})$ for any base $\pui{1},\ldots,\pui{n}$ of $\Eman_{\pbs}$.
This implies that one can find a small open neighborhood $\Wman$ of $\pbs$ in $\Eman_{\pbs}$ such that $\Alin(\Wman) \subset\Uman$ and $\Alin_i(\Wman) \subset\Uman$ for all $i\in\bN$.
In particular, $\Wman\subset\Uman_{\Alin_i}$ for all $i\in\bN$ and $\Wman\subset\Uman_{\Alin}$.

Let $\px\in\Wman\subset\Uman_{\Alin_i}$ and $\omega$ be the leaf of $\Foliation$ containing $\px$.
Then $\Alin_i(\px) \in \omega$ for all $i\in\bN$.
Now if $\Alin(\px)\not\in\omega$, then
\[
    \Alin(\px) \, \in \,
    \Uman\cap (\overline{\omega}\setminus\omega) \, \subset \,
    \fSing\cap\Eman_{\pbs} \, = \,
    \{\pbs\}.
\]
Since $\Alin$ is an isomorphism of $\Eman_{\pbs}$, we have that $\px=\pbs$ and this point is the zero of $\Eman_{\pbs}$.
Hence, $\Alin(\px) = \px\in\omega$ which contradicts to the assumption.

Thus, $\Alin$ preserves leaves of $\Foliation$ in $\Wman$ and therefore in $\Uman_{\Alin}$ due to~\ref{enum:lm:Fgood_nbh_props:pres_on_V__pres_on_U}.
\end{proof}

\section{Groups of vector bundle automorphisms}\label{sect:vb_automorphisms}

\subsection{Exponential map $\exp\colon\EndES\to\GLES$}
Let $\vbp\colon\Eman\to\fSing$ be a vector bundle over a manifold $\fSing$.
As mentioned above, $\GLES$ is locally contractible.
For the proof of Theorem~\ref{th:KhM:Fol} we need to define a specific contraction ``along fiberwise one-parameter subgroups'' of some neighborhood $\GLESNbhId$ of $\id_{\Eman}$ in $\GLES$, see Corollary~\ref{cor:sect_of_t} below.

First we will define a fiberwise analogue of the exponential map $\exp\colon\MatR{n}\to\GLR{n}$.
Such a map is usually studied for principal $G$-bundles, c.f.~\cite{Kolar:RMP:2010}.
Though this construction is rather natural, the author was not able to find it in the available literature for this specific case.
For the convenience of the reader we will provide its explicit description.

\newcommand\EXP[1]{\mathbf{e}^{#1}}

Let $\phi\colon\vbp^{-1}(\Uman) \to \Uman\times\bR^{n}$ be a local trivialization of $\vbp$ over some open subset $\Uman\subset\fSing$.
Then for each $\Blin\in\EndES$ we have that $\vbp\circ\Blin=\vbp$, so in particular, $\Blin(\vbp^{-1}(\Uman))=\vbp^{-1}(\Uman)$.
Therefore, we get the following commutative diagram from the left:
\begin{align*}
&\xymatrix@C=16ex{
    \vbp^{-1}(\Uman) \ar[r]^-{\Blin} \ar[d]_-{\phi} & \vbp^{-1}(\Uman)   \ar[d]^-{\phi} \\
    \Uman\times\bR^{n} \ar[r]^-{ (\pbs,\pu) \, \mapsto \, (\pbs,\,\Blin_{\phi}(\pbs)\pu)} & \Uman\times\bR^{n}
}
&&
\xymatrix@C=24ex{
    \vbp^{-1}(\Uman) \ar[r]^-{\phi^{-1}\, \circ \, \EXP{\Blin_{\phi}} \, \circ \, \phi} \ar[d]_-{\phi} & \vbp^{-1}(\Uman)   \ar[d]^-{\phi} \\
    \Uman\times\bR^{n} \ar[r]^-{\EXP{\Blin_{\phi}}\colon (\pbs,\pu) \, \mapsto \, \left(\pbs,\,e^{\Blin_{\phi}(\pbs)}\pu\right)} & \Uman\times\bR^{n}
}
\end{align*}
where $\Blin_{\phi}\colon\Uman\to\MatR{n}$ is some $\Cinfty$ map.
Define the following map
\[
    \EXP{\Blin_{\phi}}\colon\Uman\times\bR^{n} \to \Uman\times\bR^{n},
    \qquad
    (\pbs,\pu)\mapsto \bigl(\pbs, e^{\Blin_{\phi}(\pbs)}\pu\bigr),
\]
where $e^{X}$ is the usual exponent of the matrix $X$.
This gives the upper arrow of the above diagram from the right:
\[
    \phi^{-1}\circ \EXP{\Blin_{\phi}} \circ \phi\colon \vbp^{-1}(\Uman)  \to \vbp^{-1}(\Uman),
\]
and makes that diagram commutative.
\begin{sublemma}\label{lm:exp_well_defined}
Let $\psi\colon\Uman\times\bR^{n}\to\vbp^{-1}(\Uman)$ be another trivialization of $\vbp$ over $\Uman$.
Then the following maps coincide:
\begin{equation}\label{equ:phi_eB_phi}
    \phi^{-1}\circ \EXP{\Blin_{\phi}} \circ \phi = \psi^{-1}\circ \EXP{\Blin_{\psi}} \circ \psi\colon \vbp^{-1}(\Uman)  \to \vbp^{-1}(\Uman).
\end{equation}
\end{sublemma}
\begin{proof}
Let $\gamma = \psi\circ\phi^{-1}\colon \Uman\times\bR^{n} \to \Uman\times\bR^{n}$, $(\pbs,\pu) = \bigl(\pbs, \Clin(\pbs)\pu\bigr)$, be the transition map between those charts, where $\Clin\colon\Uman\to\GLR{n}$ is some $\Cinfty$ map.
Then~\eqref{equ:phi_eB_phi} is equivalent to the identity $\EXP{\Blin_{\psi}} = \gamma \circ \EXP{\Blin_{\phi}} \circ \gamma^{-1}$.

As $\psi^{-1}\circ \Blin \circ \psi = \gamma \circ (\phi^{-1}\circ \Blin \circ \phi) \circ \gamma^{-1}$, we have that $\Blin_{\psi}(\pbs) = \Clin(\pbs) \Blin_{\phi}(\pbs)\Clin^{-1}(\pbs)$, whence
\[
    e^{\Blin_{\psi}(\pbs)} =
    e^{\Clin(\pbs) \Blin_{\phi}(\pbs)\Clin^{-1}(\pbs)} =
    \Clin(\pbs) e^{\Blin_{\phi}(\pbs)}\Clin^{-1}(\pbs),
\]
which implies that $\EXP{\Blin_{\psi}} = \gamma \circ \EXP{\Blin_{\phi}} \circ \gamma^{-1}$.
\end{proof}

Hence, for each $\Blin\in\EndES$, one can define the map $\EXP{\Blin}\in\GLES$ by the following rule: for any local trivialization $\phi\colon\vbp^{-1}(\Uman) \to \Uman\times\bR^{n}$ of $\vbp$ over some open subset $\Uman\subset\fSing$,
\[
    \EXP{\Blin}|_{\vbp^{-1}(\Uman)}:=\phi\circ \EXP{\Blin_{\phi}} \circ \phi^{-1}.
\]
Due to Lemma~\ref{lm:exp_well_defined} this definition does not depend on a particular chart $\phi$.

\begin{sublemma}\label{lm:exp_props}
Suppose that $\fSing$ is compact.
Then the following statements hold.
\begin{enumerate}[leftmargin=*, label={\rm\alph*)}]
\item\label{enum:lm:exp_props:cont}
The map $\exp\colon \EndES \to \GLES$, $\exp(\Blin) = \EXP{\Blin}$, is continuous.
\item\label{enum:lm:exp_props:homeo}
There exists a star-convex neighborhood $\NbhEnd$ of the zero element $\vbp\colon\Eman\to\fSing$ of $\EndES$ and an open neighborhood $\NbhIdE$ of $\id_{\Eman}$ in $\GLES$, such that $\exp(\NbhEnd)=\NbhIdE$ and the restriction map $\restr{\exp}{\NbhEnd}\colon\NbhEnd\to\NbhIdE$ is a homeomorphism.
\item\label{enum:lm:exp_props:contr}
The map $\Ghmt\colon\NbhIdE \times[0;1]\to \NbhIdE$, $\Ghmt(\Alin,t) = \exp( t \exp^{-1}(\Alin))$, is a strong deformation retraction of $\NbhIdE$ into the point $\id_{\Eman}\in\GLES$.
Moreover, for every $\Alin\in\NbhIdE$ the map $\Ghmt_{\Alin}\colon\Eman\times[0;1]\to\Eman$, $\Ghmt_{\Alin}(\px,t) = \Ghmt(\Alin,t)(\px)$, is $\Cinfty$.
\item\label{enum:lm:exp_props:1-subgr}
Let $\pbs\in\fSing$, $\Gy{\pbs}$ be a \term{Lie} subgroup of the group $\Aut(\Eman_{\pbs})$ with $\dim\Gy{\pbs}\geqslant1$, and $\Blin\in\NbhEnd\subset\EndES$ be such that $\exp(\Blin)|_{\Eman_{\pbs}}  \in \Gy{\pbs}$.
Then $\exp(t\Blin)|_{\Eman_{\pbs}} \in \Gy{\pbs}$ for all $t\in[0;1]$ as well.
\end{enumerate}
\end{sublemma}
\begin{proof}
Lemma easily follows from formulas for $\exp$ in local charts and the standard properties of the usual exponential map $\exp\colon\MatR{n}\to\GLR{n}$.
We leave its proof for the reader.
Also note that the image $\bigl\{ \exp(t\Blin)|_{\Eman_{\pbs}} \bigr\}_{t\in\bR}$ in~\ref{enum:lm:exp_props:1-subgr} is a one-parameter subgroup of $\Gy{\pbs}$ containing $\Alin$.
\end{proof}

The following Corollary~\ref{cor:sect_of_t} is proved in~\cite{KhokhliukMaksymenko:Nbh:2022}, and is actually based on the method by E.~Lima~\cite{Lima:CMH:1964}, and uses a specific contraction (by convex linear combinations) of some neighborhood of $\id_{\Eman}$ in $\GLES$.
We will present another proof which uses the contraction $\Ghmt$ from Lemma~\ref{lm:exp_props}\ref{enum:lm:exp_props:contr} ``along fiberwise one-parametric subgroups'', and will be exploited in the proof of Theorem~\ref{th:KhM:Fol}.
\begin{subcorollary}[{\rm c.f.~\cite{KhokhliukMaksymenko:Nbh:2022}}]\label{cor:sect_of_t}
Let $\Mman$ be a manifold, $\fSing \subset\Mman$ be a connected compact proper submanifold, $\vbp\colon\Eman\to\fSing$ a regular neighborhood of $\fSing$ in $\Mman$.
Then for every closed subset $\Xman\subset\Mman$ with $\Xman\cap\fSing=\varnothing$ there exists an open neighborhood $\NbhIdE_{\Xman}$ of $\id_{\Eman}$ in $\GLES$ and a section $\sectrVB\colon\NbhIdE_{\Xman} \to \DiffFix[\vbp]{\Mman}{\Xman\cup\fSing}$ of the homomorphism $\tFibMap$, see~\eqref{equ:map_tanh}, i.e. $\tfib{\bigl(\sectrVB(\Alin)\bigr)} = \Alin$ for all $\Alin\in\NbhIdE_{\Xman}$.
\end{subcorollary}
\begin{proof}
Let $\Ghmt\colon\NbhIdE\times[0;1]\to\NbhIdE$ be a strong deformation retraction of some neighborhood $\NbhIdE$ of $\id_{\Eman}$ in $\GLES$ into the point $\id_{\Eman}$ constructed in Lemma~\ref{lm:exp_props}\ref{enum:lm:exp_props:contr}.
Fix any continuous norm $\nrm{\cdot}\colon\Eman\to[0;+\infty]$ on the fibers of $\Eman$ being $\Cinfty$ on $\Eman\setminus\fSing$ and for $s>0$ let
\[ \OBall{s} := \{ \px\in\Eman \subset\Mman \mid \nrm{\px} \leqslant s  \} \]
be the corresponding tubular neighborhood of $\fSing$ in $\Eman$.
Since $\fSing$ is compact and $\Xman\cap\fSing=\varnothing$, there exists $\eps>0$ such that $\OBall{\eps} \subset \Mman\setminus\Xman$.

Let also $\mu:[0;+\infty)\to[0;1]$ be any $\Cinfty$ function such that $\mu=1$ on $[0;\aConst]$ and $\mu=0$ on $[\bConst;+\infty)$.
Define the following map $\sectrVB\colon\NbhIdE\to\Ci{\Mman}{\Mman}$ by
\begin{equation}\label{equ:section_restrT}
    \sectrVB(\Alin)(\px) =
        \begin{cases}
            \px, & \px\in \Mman\setminus\OBall{\bConst\eps}, \\
            \Ghmt\bigl(\Alin, \mu(\nrm{\px}/\eps)\bigr)(\px) =
            \Ghmt_{\Alin}(\px, \mu(\nrm{\px}/\eps)), & \px\in\OBall{\eps}.
        \end{cases}
\end{equation}
If $\nrm{\px} \in [\bConst\eps;\eps]$, then $\mu(\nrm{\px}/\eps))=0$, whence for every $\Alin\in\NbhIdE$ we have that
\[
    \Ghmt\bigl(\Alin, \mu(\nrm{\px}/\eps)\bigr)(\px) =
    \Ghmt(\Alin, 0)(\px) = \px.
\]
Thus, both formulas in~\eqref{equ:section_restrT} coincide on $\OBall{\eps}\setminus\OBall{\bConst\eps}$, and therefore $\sectrVB(\Alin)$ is a $\Cinfty$ map, i.e.\ $\sectrVB$ is well-defined.
Furthermore,
\begin{enumerate}[label={\rm(\alph*)}]
\item\label{enum:sgl:sect}
for all $\Alin\in\NbhIdE$ and $\px\in\overline{\OBall{\aConst\eps}}$ we have that
\[
    \sectrVB(\Alin)(\px) =
    \Ghmt(\Alin, \mu(\nrm{\px}/\eps))(\px) =
    \Ghmt(\Alin, 1)(\px) = \Alin(\px);
\]
\item\label{enum:sgl:supp}
$\sectrVB(\Alin)(\px)=\px$ for all $\Alin\in\NbhIdE$ and $\px\in\Mman\setminus\OBall{\bConst\eps} \supset \Mman \setminus \Xman$;

\item\label{enum:sgl:id}
$\sectrVB(\id_{\Eman}) = \id_{\Mman}$.
\end{enumerate}

One easily checks that $\sectrVB$ is continuous from the weak $\Cinfty$ topology on $\NbhIdE$ to weak $\Cinfty$ topology on $\Ci{\Mman}{\Mman}$.
On the other hand, even if $\Mman$ is non-compact, the image of $\sectrVB$ consists of maps supported in the same compact neighborhood $\OBall{1}$ of $\fSing$.
Therefore, $\sectrVB$ is also continuous between the corresponding strong $\Cinfty$ topologies.

Since $\Diff(\Mman)$ is open in $\Ci{\Mman}{\Mman}$ and $\sectrVB(\id_{\Eman}) = \id_{\Mman} \in\Diff(\Mman)$, by~\ref{enum:sgl:id}, we see that the set $\NbhIdE_{\Xman} := \sectrVB^{-1}(\Diff(\Mman)) \subset \NbhIdE$ is an open neighborhood of $\id_{\Eman}$ in $\GLES$.
Then~\ref{enum:sgl:sect} and~\ref{enum:sgl:supp} imply that for each (vector bundle morphism) $\Alin\in\NbhIdE_{\Xman}$ its image $\sectrVB(\Alin)$ is a diffeomorphism of $\Mman$ fixed on $\Xman$ and coinciding with $\Alin$ near $\fSing$, so $\tFibMap\circ\sectrVB(\Alin) = \Alin$.
Hence, $\sectrVB(\NbhIdE_{\Xman}) \subset \DiffFix[\vbp]{\Mman}{\Xman\cup\fSing}$ and $\tFibMap\circ\sectrVB = \id_{\NbhIdE_{\Xman}}$, i.e.\ $\sectrVB$ is a section of $\tFibMap$ on $\NbhIdE_{\Xman}$.
\end{proof}

\subsection{Local contractibility of $\GLESGG{\GDistr}$}
Say that a topological space $\Xman$ is \term{locally contractible at a point $\px\in\Xman$}, if $\px$ has a local base of the topology of $\Xman$ consisting of (open) contractible neighborhoods.
Also, $\Xman$ is \term{locally contractible} if it is locally contractible at each $\px\in\Xman$.

\begin{sublemma}\label{lm:char:0-uniform}
Let $\GDistr$ be a Lie groups distribution over $\Yman\subset\fSing$.
Then the following conditions are equivalent:
\begin{enumerate}
\item\label{enum:lm:char:0-uniform:ker_open} $\zeroker{\GDistr}$ is open;
\item\label{enum:lm:char:0-uniform:contr_at_id} $\GLESGG{\GDistr}$ is locally contractible at $\id_{\Eman}$;
\item\label{enum:lm:char:0-uniform:loc_contractible} $\GLESGG{\GDistr}$ is locally contractible.
\end{enumerate}
\end{sublemma}
\begin{proof}
The implication~\ref{enum:lm:char:0-uniform:contr_at_id}$\Rightarrow$\ref{enum:lm:char:0-uniform:ker_open} follows from the implication
\ref{enum:lm:0-uniform:contr}$\Rightarrow$\ref{enum:lm:0-uniform:ker0} of Lemma~\ref{lm:G_contractible}.

The implication \ref{enum:lm:char:0-uniform:loc_contractible}$\Rightarrow$\ref{enum:lm:char:0-uniform:contr_at_id} is evident, while the inverse \ref{enum:lm:char:0-uniform:contr_at_id}$\Rightarrow$\ref{enum:lm:char:0-uniform:loc_contractible} follows from the fact that $\GLESGG{\GDistr}$ is a topological group, and therefore it is homogeneous, so it has ``the same topology at each point''.

\ref{enum:lm:char:0-uniform:ker_open}$\Rightarrow$\ref{enum:lm:char:0-uniform:contr_at_id}.
We have that each $\Gy{\pbs}$ is a Lie subgroup of $\Aut(\Eman_{\pbs})$ and $\zeroker{\GDistr}$ is open in $\GLESGG{\GDistr}$.
It is necessary to prove that every open neighborhood $\NbhGIdE$ of $\id_{\Eman}$ in $\GLESGG{\GDistr}$ contains some contractible neighborhood $\NbhGIdEp$.

Since $\zeroker{\GDistr}$ is open in $\GLESGG{\GDistr}$ one can assume that
\[
    \NbhGIdE = \UU \cap \GLESGG{\GDistr} \subset \zeroker{\GDistr}
\]
for some open subset $\UU \subset \GLES$.

By Lemma~\ref{lm:exp_props}\ref{enum:lm:exp_props:homeo}, the map $\exp:\EndES\to\GLES$ induces a homeomorphism of some neighborhood $\NbhEnd$ of $\vbp$ in $\EndES$ onto an open neighborhood $\NbhIdE$ of $\id_{\Eman}$ in $\GLES$.
Then one can choose an open star-convex neighborhood $\NbhEndp \subset \NbhEnd$ of $\vbp$ in $\EndES$ such that $\exp(\NbhEndp) \subset \NbhIdE \cap \UU$.
We claim that the following neighborhood $\NbhGIdEp$ of $\id_{\Eman}$ in $\GLESGG{\GDistr}$:
\[
    \NbhGIdEp := \exp(\NbhEndp) \cap \GLESGG{\GDistr} \subset
    \NbhIdE \cap \UU \cap \GLESGG{\GDistr} \subset \NbhIdE \cap \zeroker{\GDistr}
\]
is invariant under the contraction $\Ghmt$ from Lemma~\ref{lm:exp_props}\ref{enum:lm:exp_props:contr}, i.e.\ $\Ghmt\bigl(\NbhGIdEp\times[0;1]\bigr)\subset \NbhGIdEp$.
In other words, for every $\Blin\in\exp^{-1}(\NbhGIdEp)$, we have that $\exp(t\Blin)\in\GLESGG{\GDistr}$ for all $t\in[0;1]$.

Indeed, let $\pbs\in\fSing$, so $\restr{\exp(\Blin)}{\Eman_{\pbs}} \in \Gy{\pbs}$.
If $\dim\Gy{\pbs} \geqslant 1$, then $\exp(t\Blin)|_{\Eman_{\pbs}} \in \Gy{\pbs}$ for all $t\in[0;1]$ due to Lemma~\ref{lm:exp_props}\ref{enum:lm:exp_props:1-subgr}.

On the other hand, if $\dim\Gy{\pbs}=0$, then $\exp(\Blin)=\id_{\Eman_{\pbs}}$, since $\exp(\Blin)\in\NbhGIdEp\subset\zeroker{\GDistr}$.
Hence, $\exp(t\Blin)|_{\Eman_{\pbs}}=\id_{\Eman_{\pbs}} \in \Gy{\pbs}$ for all $t\in[0;1]$ as well.
\end{proof}

\subsection{Lie groups distribution with non-open $0$-kernel.}
Let $\fSing = [0;1]$, and $\vbp\colon\bR^2\times\fSing\to\fSing$, $\vbp(\pv,\pbs)=\pbs$, be the trivial vector bundle.
Then we have a natural identification of $\GL(\bR^2\times\fSing,\fSing)$, the space of vector bundle automorphisms of $\vbp$ fixed on $\fSing=0\times\fSing$, with the space of $\Cinfty$ maps $\Ci{\fSing}{\GLR{2}}$, so that the identity morphism $\id_{\Eman}$ corresponds to the constant function $\eps\colon\fSing\to\GLR{2}$ into the unit matrix $I=\amatr{1}{0}{0}{1}$.

Denote by
\[
    R_{\phi} = \amatr{\cos(2\pi\phi)}{\sin(2\pi\phi)}{-\sin(2\pi\phi)}{\cos(2\pi\phi)} \in \GLR{2},
    \quad \phi\in\bR,
\]
the matrix of rotation of the plane by the angle $2\pi\phi$.
Then for every $\Cinfty$ function $\mu\colon\fSing\to\bR$ one can define the following $\Cinfty$ vector bundle automorphism $\Alin_{\mu}\in\GL(\bR^2,\fSing)$ given by the formula $\Alin_{\mu}(\pv,\pbs) = (R_{\mu(\pbs)}, \pbs)$.

Let also $\Yman = \{0\} \cup \{1/n \mid n\in\bN \} \subset \fSing$,
\begin{align*}
    \Gy{0}   &= \GLR{2}, &
    \Gy{1/n} &=
        \left\{
            R_{k/n} \mid k=0,1,\ldots,n-1
        \right\} \cong \bZ_{n}, \ n\in\bN,
\end{align*}
be the finite cyclic group of order $n$ generated by rotation $R_{1/n}$ by the angle $2\pi/n$.
Then $\GDistr = \{ \Gy{\pbs} \mid \pbs\in\Yman \}$ is a Lie groups distribution over the set $\Yman$.

\begin{sublemma}
$\zeroker{\GDistr}$ is \term{not open} in $\GLESGG{\GDistr}$.
\end{sublemma}
\begin{proof}
Indeed, let $\UIdProp$ be any neighborhood of the function $\eps$ in $\Ci{\fSing}{\GLR{2}}$ (corresponding to $\id_{\Eman}$).
By the definition of $\Cinfty$ topology, one can assume (decreasing $\UIdProp$ if necessary) that there exists  $\delta>0$ such that
\[
    \UIdProp = \{ A\colon\fSing\to\GLR{2} \mid \nrm{A(\pbs)-I} < \delta \ \text{for all} \ \pbs\in\fSing \}.
\]
Then one can find a $\Cinfty$ function $\mu\colon\fSing\to\bR$ such that
\begin{enumerate}[label={\rm(\alph*)},leftmargin=*]
\item $\Alin_{\mu}\in\UIdProp$, i.e.\ $\nrm{R_{\mu(\pbs)} - I} < \delta$ for all $\pbs\in\bR$;
\item there exists $n>0$ such that $\mu(\pbs)=\pbs$ for $\pbs\in[0;\tfrac{1}{n}]$ for some $n>0$, and $\mu(\pbs)=0$ for $\pbs\in[\tfrac{1}{n+1};1]$.
\end{enumerate}
Notice that
\begin{enumerate}[label={\rm(\roman*)},leftmargin=*]
\item
if $m<n$, i.e.\ $1/n < 1/m$, then $\mu(1/m) = 0$, whence $\restr{\Alin_{\mu}}{\bR^2\times\frac{1}{m}} = R_{0} = I \in \Gy{1/m}$;
\item
on the other hand, if $m\geqslant n$, then $\mu(1/m) = 1/m$, whence $\restr{\Alin_{\mu}}{\bR^2\times\frac{1}{m}} = R_{2\pi/m} \in \Gy{1/m}$.
\end{enumerate}
Thus, $\dif\in\UIdProp\cap\GLESGG{\GDistr}$, but $\restr{\Alin_{\mu}}{\bR^2\times\frac{1}{m}}\neq  \id_{(\bR^2\times\frac{1}{m})}$ when $m\geqslant n$, though $\dim\Gy{1/m}=0$.
This proves that $\zeroker{\GDistr}$ is not open.
\end{proof}

\section{Proof of Theorem~\ref{th:KhM:Fol}}\label{sect:proof:th:KhM:Fol}
Let $\Mman$ be a manifold, $\fSing \subset\Mman$ be a connected compact proper submanifold, $\vbp\colon\Eman\to\fSing$ a regular neighborhood of $\fSing$ in $\Mman$, $\Foliation$ a partition of $\Mman$ for which $\fSing$ is $\Foliation$-saturated.

\begin{enumerate}[wide, itemsep=1ex]
\item[\ref{enum:th:KhM:Fol:HE}]
Suppose $\fSing$ has an \NGOOD{$\Foliation$}\ open neighborhood $\Uman\subset\Eman$.
Then it is shown in~\cite{KhokhliukMaksymenko:Nbh:2022} that the deformation $\Hhmt$ of $\DiffInv{\Mman}{\fSing}$ into $\DiffInv[\vbp]{\Mman}{\fSing}$ constructed in Theorem~\ref{th:isot_nbh}\ref{th:isot_nbh:linearization}, see item~\ref{enum:rem:th:isot_nbh:proof_1} of the proof, also preserves the subgroup $\Diff(\Foliation)$ and therefore its intersections with groups~\eqref{equ:DMS_homot_equiv}.
This gives a deformation of the bottom triple from~\eqref{equ:DMS_homot_equiv_fol} into the corresponding top triple.

\item[\ref{enum:th:KhM:Fol:Fibr}]
Suppose further that for each $\pbs\in\fSing$ the group $\Gy{\pbs} = \{ \restr{(\tfib{\dif})}{\Eman_{\pbs}} \mid \dif\in\DiffFix{\Foliation}{\fSing} \}$ from~\eqref{equ:GDistroFol} is closed in $\Aut(\Eman_{\pbs})$ and for the distribution of groups $\GFolDistr=\{\Gy{\pbs}\}_{\pbs\in\fSing}$ its $0$-kernel $\zeroker{\GFolDistr}$ is open in $\GLESGG{\GFolDistr}$.

\term{We will find a neighborhood $\UU$ of $\id_{\Eman}$ in $\GLESGG{\GFolDistr}$ and a continuous map
\[
    \sectrVB\colon \UU \to \DiffFix[\vbp]{\Foliation}{\fSing} \subset \DiffFix{\Foliation}{\fSing}
\]
such that $\tFibMap\circ\sectrVB=\id_{\UU}$.}
This will give at once sections of the two upper arrows $\tFibMap$ from~\eqref{equ:diagam_homot_eq_all_maps:fol}.

Let $\NbhEnd$ be a star-convex neighborhood of $\vbp\colon\Eman\to\fSing$ (as the zero of the algebra $\EndES$) such that $\NbhIdE:=\exp(\NbhEnd)$ and the restriction $\exp\colon\NbhEnd \to \NbhIdE$ is a homeomorphism.
Let also $\Ghmt\colon\NbhIdE\times[0;1]\to\NbhIdE$, $\Ghmt(\Alin,t)=\exp(t\exp^{-1}(\Alin))$, be a contraction of some neighborhood $\NbhIdE$ of $\id_{\Eman}$ in $\GLES$ defined in Lemma~\ref{lm:exp_props}\ref{enum:lm:exp_props:contr}.

Since each $\Alin\in\GLESGG{\GFolDistr}$ is fixed on $\fSing$, there exists an open neighborhood $\Vman\subset\Uman$ of $\fSing$ with compact closure and a neighborhood $\NbhGIdE \subset \NbhIdE \cap \GLESGG{\GFolDistr}$ of $\id_{\Eman}$ in $\GLESGG{\GFolDistr}$ such that $\Alin(\overline{\Vman}) \subset \Uman$ for all $\Alin\in\NbhGIdE$.
As $\Uman$ is \NGOOD{$\Foliation$}\, we have by Lemma~\ref{lm:Fgood_nbh_props}\ref{enum:lm:Fgood_nbh_props:clGLESGG} that $\Alin$ preserves the leaves of $\GFolDistr$ on $\Alin^{-1}(\Uman) \cap \Uman$.

Also, since $\zeroker{\GFolDistr}$ is open in $\GLESGG{\GFolDistr}$, we have by Lemma~\ref{lm:char:0-uniform} that there exist a neighborhood $\NbhGIdEp\subset \NbhGIdE$ of $\id_{\Eman}$ in $\GLESGG{\GFolDistr}$ being invariant under $\Ghmt$, i.e.\ $\Ghmt(\NbhGIdEp\times[0;1])\subset\NbhGIdEp$.

Put $\Xman:=\Mman\setminus\Vman$.
Then by Corollary~\ref{cor:sect_of_t}, there is an open neighborhood $\NbhIdE_{\Xman} \subset \NbhIdE$ of $\id_{\Eman}$ in $\GLES$ and a continuous map $\sectrVB\colon \NbhIdE_{\Xman} \to \DiffFix[\vbp]{\Mman}{\Xman\cup\fSing}$ such that $\tfib{\bigl(\sectrVB(\Alin)\bigr)} = \Alin$ for all $\Alin\in\NbhIdE_{\Xman}$, i.e.\ $\tFibMap\circ\sectrVB=\id_{\NbhIdE_{\Xman}}$.

Then $\UU:= \NbhIdE_{\Xman} \cap \NbhGIdEp$ is also a neighborhood of $\id_{\Eman}$ in $\GLESGG{\GFolDistr}$.
We claim that
\[
    \sectrVB(\UU) \subset \Diff(\Foliation),
\]
which will imply that in fact
\[
    \sectrVB(\UU) \subset \Diff(\Foliation) \cap \DiffFix[\vbp]{\Mman}{\Xman\cup\fSing} =
    \DiffFix[\vbp]{\Foliation}{\Xman\cup\fSing}
    \subset
    \DiffFix[\vbp]{\Foliation}{\fSing}.
\]
In particular, the restriction $\restr{\sectrVB}{\UU}\colon \UU \to \DiffFix[\vbp]{\Foliation}{\Xman\cup\fSing}$ will be the  required local section of both $\tFibMap\colon\DiffFix[\vbp]{\Foliation}{\fSing} \to \GLESGG{\GFolDistr}$ and $\tFibMap\colon\DiffFix{\Foliation}{\fSing} \to \GLESGG{\GFolDistr}$.

Indeed, let $\Alin\in\UU$, $\px\in\Mman$, and $\omega$ be the leaf of $\Foliation$ containing $\px$.
We have to show that $\sectrVB(\Alin)(\px)\in\omega$.
If $\px\in\Mman\setminus\OBall{\bConst\eps}$, then $\sectrVB(\Alin)(\px)=\px\in\omega$.

Suppose $\px\in\OBall{\bConst\eps} \subset\Vman$ and let $\pbs = \vbp(\px) \in \Eman$.
By the construction we have that
\begin{itemize}
\item $\Ghmt(\Alin,t)\in \NbhGIdEp \subset \GLESGG{\GFolDistr}$ for all $t\in[0;1]$, i.e.\ $\Ghmt(\Alin,t)|_{\Eman_{\pbs}} \in \Gy{\pbs}$ for all $\pbs\in\fSing$;
\item
$\Ghmt(\Alin,t)\in \NbhGIdE$, and therefore it preserves the leaves of $\Foliation$ on $\Vman$;
\item
$\sectrVB(\Alin)(\px)=\Ghmt(\Alin, \tau)(\px)$, where $\tau=\mu(\nrm{\px}/\eps)$.
\end{itemize}
Thus, $\sectrVB(\Alin)(\px) = \gamma(\px)$, where $\gamma = \Ghmt(\Alin, \tau) \in \Gy{\pbs}$.
Hence, $\gamma(\px)\in\omega$.

\item[\ref{enum:th:KhM:Fol:Restr}]
Finally, suppose that the right vertical arrow $\tRestr\colon\DiffInv{\Foliation}{\fSing}\to\Diff(\fSing)$ in~\eqref{equ:diagam_homot_eq_all_maps:fol} has a local section $s\colon\UU\to\DiffInv{\Foliation}{\fSing}$ defined on some neighborhood $\VV$ of $\id_{\fSing}$ in $\Diff(\fSing)$.

By Theorem~\ref{th:isot_nbh} the deformation $\Hhmt$ of $\DiffInv{\Foliation}{\fSing}$ into $\DiffInv[\vbp]{\Foliation}{\fSing}$ commutes with $\tFibMap$ in the sense that $\tfib{\Hhmt(\dif,t)}=\tfib{\dif}$ for all $\dif\in\DiffInv{\Mman}{\fSing}$ and $t\in[0;1]$, see item~\ref{enum:H:prop:5} in Section~\ref{sect:rem:proof:th:isot_nbh}\ref{enum:rem:th:isot_nbh:proof_1}.
Moreover, $\Hhmt_0\bigl(\DiffInv{\Foliation}{\fSing}\bigr)\subset \DiffInv[\vbp]{\Foliation}{\fSing}$.
Then $\Hhmt_0\circ s\colon \UU \to \DiffInv[\vbp]{\Foliation}{\fSing}$ is a local section of the left arrow $\tRestr$.
\end{enumerate}

Theorem~\ref{th:KhM:Fol} is completed.

\section{Applications}\label{sect:applications}
\subsection{Scalable foliation near extreme submanifold}\label{sect:applications_claims}
As the first application of the ``linearization part'' of Theorem~\ref{th:KhM:Fol} we will prove a ``local'' statement about ``scalable'' foliations (given by fiberwise homogeneous functions) near extreme critical submanifolds.

Let $\vbp\colon\Eman\to\fSing$ be a vector bundle over a compact manifold $\fSing$, and $\func\colon\Eman\to[0;+\infty)$ be a continuous  non-negative function such that
\begin{enumerate}[label={(\alph*)}]
    \item\label{enum:hfunc:homog}  $\func$ is \term{$k$-homogeneous} of fibers for some (possibly fractional)$k>0$;
    \item\label{enum:hfunc:zero}   $\func^{-1}(0)=\fSing$;
    \item\label{enum:hfunc:smooth} the restriction $\restr{\func}{\Eman\setminus\fSing}\colon\Eman\setminus\fSing \to (0;+\infty)$ is $\Cinfty$;
\end{enumerate}
Denote by $\hFoliation$ the partition of $\Eman$ into path components of level sets $\func^{-1}(\rr)$, $\rr\in[0;\infty)$, of $\func$.
Let also $\GLV{\hFoliation}$ be the subgroup of $\GLE$ consisting of $\hFoliation$-leaf preserving vector bundle automorphisms of all $\Eman$, and $\GLVFix{\hFoliation}{\fSing} := \GLV{\hFoliation} \cap \GLES$ be the subgroup fixed on $\fSing$, i.e.\ preserving fibers of $\vbp$.

Consider the following $\hFoliation$-saturated subset $\ATor := \func^{-1}\bigl([0;1]\bigr)$ of $\Eman$.
It follows from the conditions on $\func$ that $\ATor$ is a compact submanifold with boundary $\dATor=\func^{-1}(1)$ being a $\Cinfty$ submanifold of $\Eman$.
Let $\Foliation := \restr{\hFoliation}{\ATor}$ be the restriction of $\hFoliation$ to $\ATor$, i.e.\ the partition of $\ATor$ into path components of level sets $\func^{-1}(\rr)$, $\rr\in[0;1]$, of $\func$.
As usual denote by $\Diff(\Foliation)$ the group of $\Foliation$-leaf preserving diffeomorphisms of $\ATor$, and by $\DiffFix{\Foliation}{\fSing}$ its subgroup fixed on $\fSing$.

\newcommand\tsect{s}

Recall also that we have a continuous homomorphism $\tFibMap\colon\Diff(\Foliation)\to\GLE$, $\tFibMap(\dif)=\tfib{\dif}$, with kernel $\DiffFix[1]{\Foliation}{\fSing}$.
\begin{sublemma}\label{lm:sect_of_tFib}
$\tFibMap\bigl(\Diff(\Foliation)\bigr) = \GLV{\hFoliation}$ and the induced epimorphism $\tFibMap\colon\Diff(\Foliation)\epiArrow\GLV{\hFoliation}$ admits a section homomorphism $\tsect\colon\GLV{\hFoliation}\to\Diff(\Foliation)$, $\tsect(\Alin)=\restr{\Alin}{\ATor}$.
Hence, by Lemma~\ref{lm:principal_fibrations}\ref{enum:lm:principal_fibrations:global_sect}, we have the following commutative diagram:
\[
\xymatrix@C=5em{
    \DiffFix[1]{\Foliation}{\fSing}
        \ar@{=}[d]
        \ar@{^(->}[r]^-{\dif \,\mapsto\, (\dif,\id_{\Eman})}
        &
    \DiffFix[1]{\Foliation}{\fSing} \times \GLV{\hFoliation}
        \ar@{->>}[r]^(0.61){(\dif,\Alin) \,\mapsto\, \Alin}
        \ar[d]^-{\cong}_-{\xi\colon (\dif,\Alin) \,\mapsto\, \dif\circ(\restr{\Alin}{\ATor})}
        &
    \GLV{\hFoliation} \ar[d] \\
    \DiffFix[1]{\Foliation}{\fSing}  \ar@{^(->}[r] &
    \Diff(\Foliation)
        \ar@{->>}[r]^-{\tFibMap}
        &
    \GLV{\hFoliation}
    }
\]
in which rows are short exact sequences and $\xi$ is a homeomorphism.
\end{sublemma}
\begin{proof}
1) Let us show that $\tFibMap\bigl( \Diff(\Foliation) \bigr) \subset \GLV{\hFoliation}$.
Let $\dif\in\Diff(\Foliation)$, so $\dif$ preserves connected components of level sets of $\func$, and therefore $\func\circ\dif(\px)=\func(\px)$ for all $\px\in\ATor$.
Since $\func$ is $k$-homogeneous, it follows that $\func\circ\tfib{\dif}=\func$.
Moreover, by Lemma~\ref{lm:Fgood_nbh_props}\ref{enum:lm:Fgood_nbh_props:clGLESGG}, $\tfib{\dif}$ also preserves connected components of level sets of $\func$, i.e.\ the leaves of $\hFoliation$, and thus $\tFibMap(\dif) = \tfib{\dif}\in\GLV{\hFoliation}$.

2) Now consider the map $\tsect\colon \GLV{\hFoliation} \to \Diff(\Foliation)$, $\tsect(\Alin) = \restr{\Alin}{\ATor}$.
Since $\ATor$ is $\hFoliation$-saturated, each $\Alin\in\GLV{\hFoliation}$ leaves invariant $\ATor$ and preserves the leaves of $\Foliation$.
Hence we get a natural inclusion $\tsect:\GLV{\hFoliation} \to \Diff(\Foliation)$, $\tsect(\Alin) = \restr{\Alin}{\ATor}$.
Moreover, since $\Alin$ is a vector bundle automorphism, we have that
\[
    \tfib{(\tsect(\Alin))}(\pv) = \tfib{(\restr{\Alin}{\ATor})}(\pv) = \lim\limits_{t\to0} \tfrac{1}{t}\Alin(t\pv) = \Alin(\pv)
\]
for all $\pv\in\Eman$.
In other words, $\tFibMap\circ\tsect = \id_{\GLV{\hFoliation}}$, i.e.\ $\tsect$ is a section of $\tFibMap$.
This also implies the inverse inclusion $\GLV{\hFoliation} \subset \tFibMap\bigl(\Diff(\Foliation)\bigr)$, and thus proves the identity $\GLV{\hFoliation} = \tFibMap\bigl(\Diff(\Foliation)\bigr)$ and the existence of all the diagram.
\end{proof}

Consider the following subgroups of $\Diff(\Foliation)$:
\begin{align*}
    \GLV{\Foliation}            &:= \tsect(\GLV{\hFoliation}) = \{ \restr{\Alin}{\ATor} \mid \Alin\in\GLV{\hFoliation} \}, \\
    \GLVFix{\Foliation}{\fSing} &:= \tsect(\GLV{\hFoliation,\fSing}) = \GLV{\Foliation} \cap \DiffFix{\Foliation}{\fSing}.
\end{align*}

\begin{sublemma}\label{lm:DF_to_GF_hom_type}
In the inclusion of triples
\begin{equation}\label{equ:extr_subm:homot_equiv}
\begin{array}{lcr}
\bigl(
    \GLV{\Foliation}, \ &
    \GLVFix{\Foliation}{\fSing}, \ &
    \{ \id_{\ATor} \}
\bigr)  \\
& \cap & \\
\bigl(
    \Diff(\Foliation), &
    \DiffFix{\Foliation}{\fSing}, &
    \DiffFix[1]{\Foliation}{\fSing}
\bigr)
\end{array}
\end{equation}
the upper triple is a strong deformation retract of the lower one, and the deformation retraction is given by the map $\tsect\circ\tFibMap\colon\Diff(\Foliation)\to\GLV{\Foliation}$.
In particular, the group $\DiffFix[1]{\Foliation}{\fSing}$ is contractible, and the map of triples
\[
    \tFibMap\colon
    \bigl(
        \Diff(\Foliation),
        \DiffFix{\Foliation}{\fSing},
        \DiffFix[1]{\Foliation}{\fSing}
    \bigr)
    \to
    \bigl(
        \GLV{\hFoliation},
        \GLVFix{\hFoliation}{\fSing},
        \{ \id_{\Eman} \}
    \bigr),
    \quad
    \tFibMap(\dif)=\tfib{\dif},
\]
is a homotopy equivalence of triples.
Moreover, $\GLVFix{\hFoliation}{\fSing} = \tFibMap\bigl(\DiffFix{\Foliation}{\fSing}\bigr) \stackrel{\eqref{equ:tDFS_GLFol}}{\subset} \GLESGG{\GFolDistr}$.
\end{sublemma}
\begin{proof}
In fact, the following map $G\colon\Diff(\Foliation)\times[0;1]\to\Diff(\Foliation)$,
\[
G(\dif,t)(\px) =
\divhom_{\dif}(\px,t) =
\begin{cases}
\tfrac{1}{t}\dif(t\px), & t\in(0;1] \\
\tfib{\dif}(\px),       & t = 0
\end{cases},
\qquad \px\in\ATor,
\]
is the required strong deformation retraction, see Lemma~\ref{lm:Hadamard_for_vb}.
It is clear that $G(\dif,1)=\dif$ and $G(\dif,0)=\tfib{\dif} = \tsect\circ\tFibMap(\dif)$.
We leave the verification of all necessary properties of $G$ and all other statements of the lemma for the reader, c.f.\ \cite[Theorem~3.1.2]{KhokhliukMaksymenko:JHRS:2023}.
\end{proof}

Lemma~\ref{lm:DF_to_GF_hom_type} reduces the computation of the homotopy types of the groups $\Diff(\Foliation)$ and $\DiffFix{\Foliation}{\fSing}$ to the homotopy types of groups $\GLV{\hFoliation}$ and $\GLVFix{\hFoliation}{\fSing}$.
It is actually a variation of~\cite[Theorem~3.1.2]{KhokhliukMaksymenko:JHRS:2023}, and it describes the simplest (in a certain sense) case of the above-mentioned ``linearization'' procedure.
Its proof is also an extension of standard formulas for the isotopies between regular neighborhoods.

\subsection{Trivial vector bundle}
The advantage of passing from $\Diff(\Foliation)$ to $\GLV{\Foliation}$ is that the homotopy type of the latter group can be computed by purely homotopical methods.
At least when $\vbp$ is a trivial fibration, and the restriction of $\func$ to each fiber is ``the same'', such computations are relatively easy and are done in this Subsection.

Let $S^{n-1} \subset\bR^{n}$ be the unit sphere centered at the origin.
Fix a continuous function $\gfunc\colon\bR^{n}\to[0;+\infty)$ such that
\begin{enumerate}[label={\rm(\alph*$'$)}]
\item\label{enum:gfunc:homog}  $\gfunc$ is $k$-homogeneous for some (possibly fractional) $k>0$;
\item\label{enum:gfunc:zero}   $\gfunc^{-1}(0)=0\in\bR^{n}$.
\item\label{enum:gfunc:smooth} $\restr{\gfunc}{\bR^{n}\setminus0}\colon\bR^{n}\setminus 0 \to (0;+\infty)$ is $\Cinfty$.
\end{enumerate}
\begin{sublemma}
For every $t>0$ the map
\[
    \phi_{t}\colon S^{n-1} \to \gfunc^{-1}(t),
    \qquad
    \phi_t(\px) = \bigl(\tfrac{t}{\gfunc(\px)}\bigr)^{1/k} \px,
\]
is diffeomorphism.
In particular, $\gfunc^{-1}(t)$ is path connected for $n\geqslant2$ and consists of two points for $n=1$.

Let $\LDiff(\gfunc)$ be the subgroup of $\GLR{n}$ preserving \term{connected components} of level sets of $\gfunc$.
Then $\LDiff(\gfunc)$ is closed in $\GLR{n}$, and therefore it is a Lie group.
\end{sublemma}
\begin{proof}
The first statement is straightforward.

Consider the group $\LDiff(\gfunc)$.
If $n\geqslant2$, then each level set of $\gfunc$ is connected being diffeomorphic to $S^{n-1}$, whence $\LDiff(\gfunc)$ is closed by Lemma~\ref{lm:pres_func}.

Let $n=1$.
Denote $a=\gfunc(1)$ and $b=\gfunc(-1)$.
Notice that since we require $\gdif$ to be $\Cinfty$ only on $\bR\setminus 0$ the numbers $a,b$ may be distinct.
Then it is easy to see that
\begin{align*}
   \gdif(\pv)&:=
   \begin{cases}
        a\pv^k,    & \pv\geqslant0, \\
        b(-\pv)^k, & \pv<0.
   \end{cases}
   &
   \LDiff(\gfunc)&:=
   \begin{cases}
        \{ \pm\id_{\bR} \}, & a\neq b, \\
        \{ \id_{\bR} \},    & a=b.
   \end{cases}
\end{align*}
Hence, $\LDiff(\gfunc)$ is also closed in $\GLR{1}=\bR\setminus0$ for $n=1$.
\end{proof}

Now let $\fSing$ be a compact manifold, $\vbp\colon\Eman=\bR^{n}\times\fSing\to\fSing$ be a trivial vector bundle, and $\func\colon\Eman\to[0;\infty)$ be the function defined by $\func(\pv,\pbs) = \gfunc(\pv)$.
Then the conditions~\ref{enum:gfunc:homog}-\ref{enum:gfunc:smooth} on $\gfunc$ imply that $\func$ satisfies the conditions~\ref{enum:hfunc:homog}-\ref{enum:hfunc:smooth} from Subsection~\ref{sect:applications_claims}.
As above, denote $\ATor = \func^{-1}([0;1]) = \gfunc^{-1}([0;1])\times\fSing$ and let $\hFoliation$ be the partition of $\Eman$ into connected components of level sets of $\func$, and $\Foliation = \restr{\hFoliation}{\ATor}$ be the restriction of $\hFoliation$ to $\ATor$.
Evidently, each leaf of $\hFoliation$ is of the form $\omega\times\fSing'$, where $\omega$ is a path component of some level set of $\gfunc$ and $\fSing'$ is a path component of $\fSing$.

By Lemma~\ref{lm:DF_to_GF_hom_type} the map $\tFibMap\colon\Diff(\Foliation)\to\GLV{\hFoliation}$ is a homotopy equivalence.
The following statement simplifies the study of the homotopy type of $\GLV{\hFoliation}$.
\begin{sublemma}\label{lm:triv_vb_fol}
There are homeomorphisms
\begin{align*}
    \gamma\colon& \Ci{\fSing}{\LDiff(\gfunc)} \to \GLVFix{\hFoliation}{\fSing}, &
    &\gamma(\Phi)(\pv,\pbs) = \bigl(\Phi(\pbs)\pv, \pbs \bigr),
    \\
    \sigma\colon& \GLVFix{\hFoliation}{\fSing} \times \Diff(\fSing) \to \GLV{\hFoliation}, &
    &\sigma(\Alin, \dif)(\pv,\pbs) = \bigl( \Alin(\pv), \dif(\pbs)\bigr),
\end{align*}
for all $\Phi\in\Ci{\fSing}{\LDiff(\gfunc)}$, $\Alin\in\GLVFix{\hFoliation}{\fSing}$, and $\dif\in\Diff(\fSing)$.
In particular, due to Lemma~\ref{lm:DF_to_GF_hom_type}, we have the following homotopy equivalences:
\begin{align}\label{equ:DF_CiSLg}
    &\DiffFix{\Foliation}{\fSing} \simeq \Ci{\fSing}{\LDiff(\gfunc)}, &
    &\Diff(\Foliation) \simeq \Ci{\fSing}{\LDiff(\gfunc)} \times \Diff(\fSing).
    \qedhere
\end{align}
\end{sublemma}
\begin{proof}
1) First we show surjectivity of $\gamma$, i.e.\ that for every $\Alin\in\GLVFix{\hFoliation}{\fSing}$ there exists a $\Cinfty$ map $\Phi\colon\fSing\to\LDiff(\gfunc)$ such that $\Alin(\pv,\pbs)=\bigl(\Phi(\pbs)\pv, \pbs \bigr)$.

Indeed, by definition $\Alin$ is a $\Cinfty$ map $\Alin\colon\bR^{n}\times\fSing\to\bR^{n}\times\fSing$ having the following properties.
\begin{enumerate}[leftmargin=*, label={\alph*)}]
\item
$\Alin$ preserves the fibers $\Eman_{\pbs} = \bR^{n} \times\{\pbs\}$ for all $\pbs\in\fSing$, and its restriction to $\Eman_{\pbs}$ is a linear isomorphism.

\item
$\Alin$ also preserves the leaves of $\hFoliation$, i.e.\ sets of the form $\omega\times\fSing'$, where $\omega$ is a path component of some level set of $\gfunc$, and $\fSing'$ is a connected component of $\fSing$.
Hence, $\restr{\Alin}{\bR^{n} \times\{\pbs\}} = \Phi(\pbs)$ preserves $\omega\times\{\pbs\}$, and thus belongs to $\LDiff(\gfunc)$.
\end{enumerate}

The first property implies that $\Alin(\pv,\pbs) = (\Phi(\pbs)\pv,\pbs)$, where $\Phi(\pbs)\in\GLR{n}$ is some non-degenerate matrix depending on $\pbs$, so we get a map $\Phi\colon\fSing\to\GLR{n}$.
Let $e_1,\ldots,e_n$ be the standard basis of $\bR^{n}$, so $e_i=(0,\ldots,1,\ldots,0)$, where $1$ stands at position $i$ and all other coordinates are zeros.
Then for each $i=1,\ldots,n$ the map
$\Alin_i:= \restr{\Alin}{e_i \times\fSing} \colon \fSing\to\bR^{n}\times\fSing$ is $\Cinfty$ (since $\Alin$ is so), and the coordinate functions of $\Alin_i(\pbs)$ constitute the $i$-th column of the matrix $\Phi(\pbs)$.
Hence, $\Phi=(\Alin_1,\ldots,\Alin_n)\colon\fSing\to\GLR{n}$ is $\Cinfty$ as well.
Furthermore, by the second condition, the image of $\Phi$ is contained in the Lie subgroup $\LDiff(\gfunc)$ of $\GLR{n}$, whence, the induced map $\Phi\colon\fSing\to\LDiff(\gfunc)$ is also $\Cinfty$.

We leave to the reader the verification that for every $\Phi\in\Ci{\fSing}{\LDiff(\gfunc)}$ the map $\gamma(\Phi)\in\GLV{\hFoliation}$, and that (due to compactness of $\fSing$) the map $\gamma$ is a homeomorphism with respect to the corresponding Whitney topologies.

2) Note that we have a natural restriction of $\fSing$ homomorphism $\tRestr\colon\GLV{\hFoliation}\to\Diff(\fSing)$, $\tRestr(\Alin) = \restr{\Alin}{\fSing}$, with kernel $\GLV{\hFoliation,\fSing}$.
Since $\vbp$ is a trivial vector bundle, $\tRestr$ admits the following global section $s:\Diff(\fSing)\to\GLV{\hFoliation}$, $s(\dif)(\pv,\pbs)=(\pv,\dif(\pbs))$, being also a homomorphism.
Now by Lemma~\ref{lm:principal_fibrations}\ref{enum:lm:principal_fibrations:global_sect}, this gives the homeomorphism $\sigma$.
\end{proof}

Thus, in this case the computation of the homotopy types of $\Diff(\Foliation)$ and $\DiffFix{\Foliation}{\fSing}$ reduces to two independent problems of computing the homotopy types of $\Diff(\fSing)$ and of the space of smooth maps $\Ci{\fSing}{\LDiff(\gfunc)}$.
Moreover, if $\gfunc$ is a non-degenerate definite quadratic form, so $\Foliation$ is Morse-Bott, then we will get a more detailed statement, see Corollary~\ref{cor:th:homtype_DFol_extr_mb-fol}.

Before formulating it let us prove a simple auxiliary lemma concerning maps from and to the circle.
Let $1\in\Circle \subset \bC$ be the unit number in the circle.
\begin{sublemma}\label{lm:some_homeq}
{\rm 1)}~Let $G$ be a compact Lie group with the identity path component $G_0$ and unit $e$, and $\Omega^{\infty}(G_0) = \Ci{(\Circle,1)}{(G_0,e)}$ be the space of $\Cinfty$ maps $\phi:\Circle\to G$ such that $\phi(1)=e$, i.e.\ the space of $\Cinfty$ loops in $G_0$ at $e$.
Then we have a homeomorphism
\[
    \theta\colon\Omega^{\infty}(G_0) \times G \to \Ci{\Circle}{G},
    \qquad
    \theta(\phi, g)(a) = \phi(\px) g.
\]

Let $\Omega(G_0) = \Cont{(\Circle,1)}{(G_0,e)}$ be the space of \term{continuous} loop in $G_0$ at $e$.
Then the inclusion $\Omega^{\infty}(G_0) \subset \Omega(G_0)$ is a weak homotopy equivalence.
In particular, we have the following homotopy equivalence: $\Ci{\Circle}{G} \simeq \Omega(G_0) \times G$.

{\rm 2)}~Also, let $\Aman$ be a connected smooth manifold, $a\in\Aman$ any point, $\Chid{\Aman}{\Circle}$ be the subspace of $\Ci{\Aman}{\Circle}$ consisting of null-homotopic maps.
Then the ``evaluation at $a$'' homomorphism
\[
    \delta\colon\Chid{\Aman}{\Circle}\to\Circle,
    \qquad
    \delta(\phi) = \phi(a),
\]
is a homotopy equivalence.
\end{sublemma}
\begin{proof}
1) The fact that $\theta$ is a homeomorphism is straightforward.
It is also well-known that inclusion $\Omega^{\infty}(G_0) \subset \Omega(G_0)$ is a weak homotopy equivalence, see~\cite{KhokhliukMaksymenko:PIGC:2020} for discussions.

2) It suffices to show that $\delta$ admits a global section and its kernel is contractible.
Then by Lemma~\ref{lm:principal_fibrations}\ref{enum:lm:principal_fibrations:global_sect}, $\delta$ will be a homotopy equivalence.

Notice that $\Chid{\Aman}{\Circle}$ is a group with respect to pointwise multiplication and $\delta$ is a surjective homomorphism.
Evidently, it admits a global section $\Circle\to\Chid{\Aman}{\Circle}$ associating to each $\tau\in\Circle$ the constant map $\Aman\to\{\tau\}\subset\Circle$ into the point $\tau$.

Further, $\ker(\delta)=\Chid{(\Aman,a)}{(\Circle,1)}$, i.e.\ it is the subset of $\Chid{\Aman}{\Circle}$ consisting of maps $\phi\colon\Aman\to\Circle$ such that $\phi(a)=1$.
Let also $\vbp\colon\bR\to\Circle$, $\vbp(t)=e^{2\pi i t}$, be the universal covering map.
Then for each $\phi\in\ker(\delta)$ there exists a unique lifting $\hat{\phi}\colon\Aman\to\bR$ such that $\hat{\phi}(a)=0$.
It is easy to see that the correspondence $\phi\mapsto\hat{\phi}$ is a homeomorphism $\ker(\delta) \cong \Ci{(\Aman,a)}{(\bR,0)}$, and the latter space is contractible.
\end{proof}

\begin{subcorollary}\label{cor:th:homtype_DFol_extr_mb-fol}
Let $\vbp\colon\Eman=\bR^{n}\times\fSing\to\fSing$ be a trivial vector bundle, and $\func\colon\bR^{n}\times\fSing\to\bR$ be the Morse-Bott function given by $\func(\pv,\pbs)=(\pv,\pv) = \nrm{\pv}^2$.
Then the homotopy types of $\Diff(\Foliation)$ and $\DiffFix{\Foliation}{\fSing}$ for different values of $n$ are presented in the following table:
\[
\begin{array}{|c|c|c|}\hline
            & \DiffFix{\Foliation}{\fSing} \simeq \GLVFix{\hFoliation}{\fSing} & \Diff(\Foliation)  \simeq \GLV{\hFoliation}  \\ \hline
    n=1     & \text{point}                                                     & \Diff(\fSing)                                \\ \hline
    n\geqslant2  & \Ci{\fSing}{\Ort(n)}                                             & \Ci{\fSing}{\Ort(n)} \times \Diff(\fSing)    \\ \hline
\end{array}
\]
If $\fSing = \Circle$, then by Lemma~\ref{lm:some_homeq}, for $n\geqslant2$ the above homotopy equivalences imply the following \term{homotopy equivalences}:
\begin{align}\label{equ:hom_eq:DFCircle}
    &\DiffFix{\Foliation}{\fSing} \simeq \Omega^{\infty}(\SO(n))\times \Ort(n),
    &
    &\Diff(\Foliation) \simeq \Omega^{\infty}(\SO(n))\times \Ort(n) \times \Ort(2).
\end{align}
\end{subcorollary}
\begin{proof}
Consider the following function $\gfunc:\bR^{n}\to[0;+\infty)$, $\gfunc(\pv) = (\pv,\pv) =\nrm{\pv}^2$.
Then $\func(\pv,\pbs)=\gfunc(\pv) = \nrm{\pv}^2$, whence we have homotopy equivalences~\eqref{equ:DF_CiSLg}.
Recall also that $\LDiff(\gfunc)$ is the group of linear automorphisms of $\bR$ fixing every connected component of each level set of $\gfunc$.
Then for $n=1$ each level set $\gfunc^{-1}(t)$, $t>0$, consists of two points $\pm\sqrt{t}$, whence $\LDiff(\gfunc)=\{\id_{\bR}\}$, which implies the first row of the table.
On the other hand, if $n\geqslant2$, then each level set of $\gfunc$ is the $(n-1)$-sphere $S^{n-1}$, whence $\LDiff(\gfunc)=\Ort(n)$ is the group of linear automorphisms preserving $\gfunc$.
This gives the second row of the table.
\end{proof}

The author is sincerely grateful to anonymous Referee of the paper for careful reading of the initial manuscript, useful comments and suggestions.


\end{document}